\documentclass[reqno]{amsart}

\usepackage{amsthm, amsmath, amssymb, amsfonts, graphicx}

\usepackage[english]{babel}

\usepackage{bbm}

\usepackage[inline]{enumitem}

\usepackage{float}
\restylefloat{table}

\usepackage{graphicx}
\usepackage{caption}
\usepackage{subcaption}

\usepackage[colorlinks=true, pdfstartview=FitV, linkcolor=blue,
            citecolor=blue, urlcolor=blue]{hyperref}
\usepackage[usenames]{color} 

\usepackage[round]{natbib}

\newtheorem{theorem}{Theorem}
\newtheorem{proposition}[theorem]{Proposition}
\newtheorem{lemma}[theorem]{Lemma}
\newtheorem{corollary}[theorem]{Corollary}

\theoremstyle{definition}

\newtheorem{example}[theorem]{Example}
\theoremstyle{remark}
\newtheorem{remark}[theorem]{Remark}

\numberwithin{equation}{section}
\numberwithin{theorem}{section}

\def\cT{\mathcal{T}}

\def\bP{\mathbb{P}}
\def\bE{\mathbb{E}}

\def\bN{\mathbb{N}}

\DeclareMathOperator{\E}{\mathbb{E}}  

\newcommand{\vep}{\varepsilon}

\usepackage{todonotes}

\makeatletter
\def\namedlabel#1#2{\begingroup
    #2%
    \def\@currentlabel{#2}%
    \phantomsection\label{#1}\endgroup
}
\makeatother

\title[Risk-sensitive optimal stopping]{Risk-sensitive optimal stopping with unbounded terminal cost function}

\author{Damian Jelito}
\address{Institute of Mathematics, Jagiellonian University, Krak\'{o}w, Poland}
\email{damian.jelito@im.uj.edu.pl}

\author{{\L}ukasz~Stettner}
\address{Institute of Mathematics, Polish Academy of Sciences, Warsaw, Poland}
\email{l.stettner@impan.pl}

\begin{document}

\begin{abstract}
In this paper we consider an infinite time horizon risk-sensitive optimal stopping problem for a Feller--Markov process with an unbounded terminal cost function. We show that in the unbounded case an associated Bellman equation may have multiple solutions and we give a probabilistic interpretation for the minimal and the maximal one. Also, we show how to approximate them using finite time horizon problems. The analysis, covering both discrete and continuous time case, is supported with illustrative examples.


\bigskip
\noindent \textbf{MSC2020 subject classifications:} 93E20, 60G40, 49J21
\end{abstract}


\maketitle
\vspace{-1cm}

\section{Introduction}

Many practical optimal control problems could be expressed in terms of optimal stopping. This includes examples in mathematical finance (American options theory, optimal asset liquidation), statistics (sequential testing), operations research, ecology; see e.g.~\cite{Shi1978,BenLio1984,CarTou2008,BauRie2011} for details. 

Typically, a characterisation of the optimal stopping time is obtained through the study of the corresponding Snell envelope of the value process; see e.g.~\cite{Kar1981} for details and~\cite{KobQue2012} for more recent contribution. Also, in the Markovian case this could be done with the help of a specific optimality Wald--Bellman equation; see e.g.~\cite{Shi1978} for a classical contribution. The existence of a solution to this equation could be obtained e.g. by value iteration argument or penalty approach, see~\cite{Ste2011}. Also, it may result from the use of viscosity techniques applied to variational inequalities; see e.g.~\cite{BenLio1984} and~\cite{DaiMen2018}. 

Risk-sensitive problems constitute a special class of general stochastic control problems (in particular, optimal stopping problems). In this case, a decision-maker tries to optimise the certainty equivalent of the exponential utility function; see~\cite{HowMat1972} and~\cite{Whi1990}. This criterion may be seen as a non-linear extension of the mean-variance (Markowitz) approach which facilitates more robust control strategies; see e.g.~\cite{BiePli2003} for a comprehensive overview. However, using risk-sensitive criterion results in multiplicative control problems that are usually more difficult to solve than their classic risk-neutral (additive) counterparts; see~\cite{Nag2007b} and~\cite{BauPop2018}.

In this paper we consider the infinite time horizon risk-sensitive optimal stopping problems
\begin{align}
u(x)&:=\inf_{\tau}\ln\mathbb{E}_x\left[\exp\left(\int_{0}^{\tau} g(X_s)ds +G(X_{\tau})\right)\right], \quad x\in E;\label{eq:opt_stop_unb2}\\
w(x)&:=\inf_{\tau}\liminf_{T\to\infty}\ln\mathbb{E}_x\left[\exp\left(\int_{0}^{\tau\wedge T} g(X_s)ds+G(X_{\tau\wedge T})\right)\right], \quad x\in E,\label{eq:opt_stop_unb}
\end{align}
where $X$ is a standard Feller-Markov process starting at $x$ from the state space $E$, while $g$ and $G$ are continuous and non-negative \textit{running cost function} and \textit{terminal cost function}, respectively. The function $g$ is assumed to be bounded while $G$ may be unbounded from above.

The map $u$ describes the value of a standard risk-sensitive optimal stopping problem. As we show in this paper, the map $w$ emerges naturally as a limit of finite horizon stopping problems. Also, the map $w$ may be seen as a version of $u$, when a decision-maker is allowed to choose only bounded stopping times. Arguably, the main contribution of this paper is the proof that both functions $u$ and $w$ are solutions to the associated optimal stopping Bellman equation. In fact, we show that $u$ and $w$ are minimal and maximal solutions to this equation, respectively, and in general we do not have an equality between $u$ and $w$. 

This paper extends the results from~\cite{JelPitSte2019a}, where the function $G$ is assumed to be bounded. In that case, it can be shown that the Bellman equation admits a unique solution, which can be used to prove continuity of the function $u\equiv w$. This result was one of the main building blocks used in~\cite{JelPitSte2019b}, where the long-run impulse control problem was analysed. In the present paper we show a more general sufficient condition for the identity $u\equiv w$. This may be used to generalise the results from~\cite{JelPitSte2019b} to the unbounded case.

In the literature, regularity properties of the optimal stopping value function were mostly studied in the context of risk-neutral (additive) stopping problems; see e.g.~\cite{BasCec2002}. In particular, this applies to non-uniqueness of a solution to the Bellman equation; see Section 2.11 in~\cite{Shi1978} and Theorem 1.13 in~\cite{PesShi2006} for classic contributions. However, the risk-sensitive case is mostly unexplored; see~\cite{Nag2007b} and~\cite{JelPitSte2019a}. Also, it should be noted that many approximative solutions to optimal stopping problems are based on numerical solutions to the Bellman equation; see e.g.~\cite{KusDup2012} for a comprehensive overview. Thus, the study on regularity properties of optimality equation is important both from theoretical and practical point of view.

The structure of this paper is as follows. In Section~\ref{S:setting} we introduce notation and assumptions used throughout this paper. Next, in Section~\ref{S:discrete} we study discrete time version of the problem. The main contribution of this part is Theorem~\ref{th:und_over_prob}, where we link the discrete time Bellman equation with the limits of suitable finite horizon stopping value functions. In Section~\ref{S:continuous} we study a continuous time version of the problem. This is used in Section~\ref{S:Bellman_cont}, where we give a characterisation of solutions to the continuous time Bellman equation; see Theorem~\ref{th:Bellman_cont} for details. Also, in Theorem~\ref{th:suff_cont} we show a condition for the uniqueness of a solution to the Bellman equation. Our results are illustrated by the examples presented in Section~\ref{S:app2}.  In particular, in Example~\ref{ex:3} we show explicit formulae for distinct solutions to the Bellman equation. Finally, in Appendix~\ref{S:proofs} we present some deferred proofs.

\section{Preliminaries}\label{S:setting}
Let $X=(X_t)_{t\geq 0}$ be a time-homogeneous continuous time standard Markov process on a filtered measurable space $(\Omega,\mathcal{F},(\mathcal{F}_t))$ with values in a locally compact separable metric space $E$. With any $x\in E$ we associate a probability measure $\mathbb{P}_x$ describing the dynamics of the process starting from $X_0=x$; see Definition 4 in \cite[Section 1.4]{Shi1978} for details. We assume that $X$ satisfies the $C_0$-Feller property, i.e.
\[
\mathcal{P}_t \mathcal{C}_0(E)\subseteq \mathcal{C}_0(E), \quad t\geq 0,
\]
where $\mathcal{P}_t$ is the corresponding transition semigroup and $\mathcal{C}_0(E)$ denotes the family of real-valued continuous functions defined on $E$, vanishing at infinity. This is a standard assumption in the stochastic control theory. In particular, it is satisfied by L\'{e}vy processes and solutions to stochastic differential equations driven by L\'{e}vy processes; see Theorem 3.1.9 and Theorem 6.7.2 in~\cite{App2009} for details.

In addition to the $C_0$-Feller property of the Markov process, we assume several properties of the cost functions. To ease the notation, for any $T\geq 0$, let us define $\zeta_T:=\sup_{t\in [0,T]}e^{G(X_t)}$. Throughout this paper we make the following Assumptions:
\begin{enumerate}
\item[(\namedlabel{cost_functions}{$\mathcal{A}1$})] (Cost functions constraints). The map $G\colon E\mapsto [0,\infty)$ is continuous and the map $g\colon E\mapsto [0,\infty)$ is continuous and bounded. Also, the map $g$ is bounded away from zero, i.e. for some $c>0$ we get $g(\cdot)\geq c>0$. 
\item[(\namedlabel{IF}{$\mathcal{A}2$})] (Integrability). For any $T\geq 0$ and $x\in E$ we get
\[
\mathbb{E}_x\left[\zeta_T\right]<\infty.
\]
\item[(\namedlabel{Feller}{$\mathcal{A}3$})] (Continuity). For any $T\geq 0$ and a continuous function $h$ satisfying $0\leq h(\cdot)\leq G(\cdot)$, we get that the map
\[
x\mapsto \mathbb{E}_x\left[\exp \left(\int_0^T g(X_s) ds + h(X_T) \right)\right]
\]
is continuous.
\end{enumerate}

Let us now comment on these conditions. 

Assumption~\eqref{IF} requires several regularity properties for the cost functions. First, it should be highlighted that while $g$ is assumed to be bounded, we allow $G$ to be unbounded from above. Also, note that the non-negativity assumption for $G$ is merely a technical normalisation. Indeed, for a generic continuous map $\tilde{G}\colon E\mapsto \mathbb{R}$ which is bounded from below, we may subtract the quantity $\inf_{y\in E}\tilde{G}(y)$ from the both sides of~\eqref{eq:opt_stop_unb2} and~\eqref{eq:opt_stop_unb} and set $G(\cdot):=\tilde{G}(\cdot)-\inf_{y\in E}\tilde{G}(y)$. Finally, note that the assumption $g(\cdot)\geq c>0$ could be used to show that stopping at infinity cannot be optimal for our problems as this leads to infinite cost.

Assumption~\eqref{IF} requires integrability for the finite time horizon and is a standard condition in the optimal stopping literature.

Assumption~\eqref{Feller} requires continuity of the specific semigroup for unbounded functions $h$. Note that from the Feller property and monotone convergence theorem we get that $x\mapsto \mathbb{E}_x\left[\exp \left(\int_0^T f(X_s) ds +h(X_T)\right)\right]$ is lower semicontinuous for any $T\geq 0$ and a continuous function $h:E\mapsto [0,\infty)$. Thus, in assumption~\eqref{Feller} we additionally require upper semicontinuity. 

Further comments on Assumptions~\eqref{IF} and~\eqref{Feller} could be found in Section~\ref{SS:assumptions}. More specifically, we show that Assumptions~\eqref{IF} and~\eqref{Feller} could be deduced from a more general condition related to the integrability of the tail of $\zeta_T$, $T\geq 0$; see~\eqref{IF_2} and the following discussion for details.

Now, let us comment on the specific forms of~\eqref{eq:opt_stop_unb2} and~\eqref{eq:opt_stop_unb}. Setting 
\[
Z_t:=\exp\left(\int_0^{t} g(X_s)ds+G(X_{t})\right), \quad t\geq 0,
\]
from quasi-left continuity of $Z$ and Fatou Lemma, for any $x\in E$ and $\mathbb{P}_x$-almost surely finite stopping time $\tau$, we get 
\begin{equation}\label{eq:convention_1_2}
\mathbb{E}_x\left[Z_{\tau}\right]=\mathbb{E}_x\left[\liminf_{T\to\infty}Z_{\tau\wedge T}\right]\leq \liminf_{T\to\infty}\mathbb{E}_x\left[Z_{\tau\wedge T}\right].
\end{equation}
Some of the results in this paper are related to the situation when there is an equality in~\eqref{eq:convention_1_2}. Let us now provide a useful characterisation of this property.
\begin{lemma}\label{lm:convention}
Let $x\in E$ and let $\tau$ be a stopping time satisfying 
\[
\mathbb{E}_x\left[\exp\left(\int_0^{\tau} g(X_s)ds+G(X_{\tau})\right)\right]<\infty.
\]
Then, the following are equivalent
\begin{enumerate}
\item\label{p:lm:convention:1} We get
\[
\liminf_{T\to\infty}\mathbb{E}_x\left[Z_{\tau\wedge T}\right] =\mathbb{E}_x\left[\liminf_{T\to\infty}Z_{\tau\wedge T}\right].
\]
\item\label{p:lm:convention:2} The family $\left\{Z_{\tau\wedge T}\right\}$, $T\geq 0$, is $\mathbb{P}_x$-uniformly integrable, i.e.
\[
\lim_{n\to\infty}\sup_{T\geq 0}\mathbb{E}_x\left[1_{\{Z_{\tau\wedge T}\geq n\}} Z_{\tau\wedge T}\right]=0.
\]
\item\label{p:lm:convention:3} We get 
\[
\liminf_{T\to\infty} \mathbb{E}_x\left[1_{\{\tau>T\}}Z_T\right]=0.
\]
\end{enumerate}
\end{lemma}
\begin{proof}
Note that the equivalence of~\eqref{p:lm:convention:1} and~\eqref{p:lm:convention:2} follows from the standard result; see e.g. Theorem 16.14 in~\cite{Bil1995} for details. Thus, it is enough to show that~\eqref{p:lm:convention:1} is equivalent to~\eqref{p:lm:convention:3}. Using the identity
\begin{equation}\label{eq:calc_conv_1}
\mathbb{E}_x\left[Z_{\tau\wedge T}\right] =\mathbb{E}_x\left[1_{\{\tau\leq T\}}Z_{\tau}\right]+\mathbb{E}_x\left[1_{\{\tau> T\}}Z_{T}\right], \quad T\geq 0,
\end{equation}
and noting that $T\mapsto 1_{\{\tau\leq T\}}Z_{\tau}$ is increasing, by monotone convergence theorem and quasi-left continuity of $(Z_t)$ we get 
\[\lim_{T\to\infty}\mathbb{E}_x\left[1_{\{\tau\leq T\}}Z_{\tau}\right]=\mathbb{E}_x\left[Z_{\tau}\right]=\mathbb{E}_x\left[\lim_{T\to \infty}Z_{\tau\wedge T}\right]<\infty;\]
note that $\mathbb{P}_x[\tau<\infty]=1$ as by the assumptions $\mathbb{E}_x\left[e^{c\tau}\right]\leq \mathbb{E}_x\left[e^{\int_0^{\tau} g(X_s)ds+G(X_{\tau})}\right]<\infty$. Thus, letting $T\to \infty$ in~\eqref{eq:calc_conv_1}, we conclude the proof.
\end{proof}

Now, observe that from~\eqref{eq:convention_1_2}, for any $x\in E$, we get 
\begin{equation}\label{eq:ineq_u_w}
u(x)\leq w(x),
\end{equation}
where $u$ and $w$ are given by~\eqref{eq:opt_stop_unb2} and~\eqref{eq:opt_stop_unb}, respectively. In the following lemma we show that $w$ may be seen as a value of the optimal stopping problem with infimum over the family of bounded stopping times. This provides an additional explanation for~\eqref{eq:ineq_u_w}.


\begin{lemma}\label{lm:w_bounded}
Let $w$ be given by~\eqref{eq:opt_stop_unb} and let $\mathcal{T}_b$ denote the family of bounded stopping times. Then, we get
\[
w(x)=\inf_{\tau \in \mathcal{T}_b}\ln\mathbb{E}_x\left[\exp\left(\int_{0}^{\tau} g(X_s)ds +G(X_{\tau})\right)\right].
\]
\end{lemma}
\begin{proof}
First, note that using boundedness of $\tau \in \mathcal{T}_b$, we get
\begin{align*}
w(x)&\leq \inf_{\tau \in \mathcal{T}_b}\liminf_{T\to\infty}\ln\mathbb{E}_x\left[e^{\int_{0}^{\tau\wedge T} g(X_s)ds +G(X_{\tau\wedge T})}\right]\\
&=  \inf_{\tau \in \mathcal{T}_b}\ln\mathbb{E}_x\left[e^{\int_{0}^{\tau} g(X_s)ds +G(X_{\tau})}\right], \quad x\in E.
\end{align*}
Second, let $x\in E$, $\varepsilon>0$, and $\tau_{\varepsilon}$ be an $\varepsilon$-optimal stopping time for $w(x)$. Then, there exists a sequence $(T_n)\subset \mathbb{R}_+$ such that $T_n\to\infty$ as $n\to\infty$ and
\begin{align*}
\inf_{\tau \in \mathcal{T}_b}\ln\mathbb{E}_x\left[e^{\int_{0}^{\tau} g(X_s)ds +G(X_{\tau})}\right] &\leq \lim_{n\to\infty}\ln\mathbb{E}_x\left[e^{\int_{0}^{\tau_{\varepsilon}\wedge T_n} g(X_s)ds +G(X_{\tau_{\varepsilon}\wedge T_n})}\right]\\
& = \liminf_{T\to\infty}\ln\mathbb{E}_x\left[e^{\int_{0}^{\tau_{\varepsilon}\wedge T} g(X_s)ds +G(X_{\tau_{\varepsilon}\wedge T})}\right]\\
& \leq w(x) +\varepsilon.
\end{align*}
Thus, letting $\varepsilon\to 0$ we get $ \inf_{\tau \in \mathcal{T}_b}\ln\mathbb{E}_x\left[e^{\int_{0}^{\tau} g(X_s)ds +G(X_{\tau})}\right]\leq w(x)$, which concludes the proof.
\end{proof}

\section{Discrete time optimal stopping}\label{S:discrete}

In this section we consider a discrete-time version of the problems~\eqref{eq:opt_stop_unb2} and~\eqref{eq:opt_stop_unb}. By $X$ we denote a standard discrete-time Markov process with values in $E$ and for simplicity we write $X=(X_n)_{n\in \mathbb{N}}$, where $\mathbb{N}:=\{0,1,2,\ldots \}$ denotes the set of non-negative integers.
It should be noted that the results in this section do not require continuity assumptions from Section~\ref{S:setting}. 

By analogy to~\eqref{eq:opt_stop_unb2} and~\eqref{eq:opt_stop_unb}, we define
\begin{align}
u(x)&:=\inf_{\tau\in \mathcal{T}_0}\ln\mathbb{E}_x\left[\exp\left(\sum_{i=0}^{\tau-1} g(X_i)+G(X_{\tau})\right)\right], \quad x\in E;\label{eq:opt_stop_unb_disc}\\
w(x)&:=\inf_{\tau\in \mathcal{T}_0}\liminf_{n\to\infty}\ln\mathbb{E}_x\left[\exp\left(\sum_{i=0}^{\tau\wedge n-1} g(X_i)+G(X_{\tau\wedge n})\right)\right], \quad x\in E,\label{eq:opt_stop_unb_disc2}
\end{align}
where $\mathcal{T}_0$ denotes the family of stopping times with values in $\mathbb{N}$ and we follow the convention $\sum_{i=0}^{-1}(\cdot)=0$. 
Also, let us define the Bellman operator
\[
Sh(x):=e^{G(x)} \wedge e^{g(x)} \mathbb{E}_x[h(X_1)], \quad  x\in E,
\]
where $h:E\mapsto \mathbb{R}_+$ is a non-negative measurable function. In this section we characterise solutions to the Bellman equation, i.e. measurable functions $v:E\mapsto \mathbb{R}_+$ satisfying
\begin{equation}\label{eq:Bellman}
e^{v(x)}=S e^v (x), \quad x\in E.
\end{equation}
More explicitly, in Theorem~\ref{th:und_over_prob} we show that $u$ and $w$ are minimal and maximal solutions to~\eqref{eq:Bellman}, respectively.

We start with finding the minimal and maximal solutions to~\eqref{eq:Bellman}. Recalling non-negativity of the functions $g$ and $G$ and~\eqref{eq:ineq_u_w}, we get 
\[
0\leq u(x)\leq w(x)\leq G(x), \quad x\in E.
\]
Based on these inequalities, to get the extremal solutions to~\eqref{eq:Bellman} we iterate the lower and upper bounds for $u$ and $w$. Thus, we define recursively the families of functions
\begin{align}
\underline{w}_0(x) & :=0, &  \underline{w}_{n+1}(x) & :=\ln S e^{\underline{w}_n}(x), & \quad n\in \mathbb{N},\, x\in E;\label{eq:underline_w}\\
\overline{w}_0(x) & :=G(x), &  \overline{w}_{n+1}(x) & :=\ln S e^{\overline{w}_n}(x), & \quad n\in \mathbb{N},\, x\in E.\label{eq:overline_w}
\end{align}
In the following proposition we show the probabilistic characterisation of the sequences $(\underline{w}_n)$ and $(\overline{w}_n)$. The proof is similar to the proof of Proposition 3 from~\cite{JelPitSte2019a}, where $G$ is assumed to be bounded from above, and therefore is omitted for brevity.

\begin{proposition}\label{pr:disc_und_over}
Let the sequences of functions $(\underline w_n)$ and $(\overline w_n)$ be given by~\eqref{eq:underline_w} and~\eqref{eq:overline_w}, respectively. Then,
\begin{enumerate}
\item For any $x\in E$, the sequence $(\underline w_{n}(x))$ is non-decreasing. Moreover, we get 
\[
e^{\underline w_{n}(x)}=\inf_{\tau\leq n} \E_x\left[e^{\sum_{i=0}^{\tau-1} g(X_i)+1_{\{\tau<n\}}G(X_\tau) }\right], \quad n\in \mathbb{N}, \, x\in E,
\]
and the optimal stopping time for $\underline w_{n}$ is given by
\begin{equation}\label{eq3'}
\underline \tau_n:=\min\left\{i\geq 0: \underline w_{n-i}(X_i)=G(X_i)\right\}\wedge n.
\end{equation}

\item For any $x\in E$,  the sequence $(\overline w_{n}(x))$ is non-increasing. Moreover we get 
\[
e^{\overline w_{n}(x)}=\inf_{\tau\leq n} \E_x\left[e^{\sum_{i=0}^{\tau-1} g(X_i)+G(X_\tau) }\right], \quad n\in \mathbb{N}, \, x\in E,
\]
and the optimal stopping time for $\overline w_{n}$ is given by
\begin{equation}\label{eq3''}
\overline \tau_n:=\min\left\{i\geq 0: \overline w_{n-i}(X_i)=G(X_i)\right\}.
\end{equation}
\end{enumerate}
\end{proposition}

Based on Proposition~\ref{pr:disc_und_over} we may define
\begin{equation}\label{eq:und_over_w}
\underline{w}(x):=\lim_{n\to\infty}\underline{w}_n(x), \quad \text{and} \quad \overline{w}(x):=\lim_{n\to\infty}\overline{w}_n(x), \quad x\in E.
\end{equation}
Using monotone convergence theorem we get that both $\underline{w}$ and $\overline{w}$ satisfy the Bellman equation~\eqref{eq:Bellman}. Also, for any measurable function $v$ solving~\eqref{eq:Bellman} and satisfying $0\leq v(x)\leq G(x)$, we iteratively get $\underline{w}_n(x)\leq v(x)\leq \overline{w}_n(x)$, $x\in E$, and consequently 
\begin{equation}\label{eq:Bellman_bounds}
\underline{w}(x)\leq v(x)\leq  \overline{w}(x), \quad x\in E.
\end{equation}
Thus, the maps $\underline{w}$ and $\overline{w}$ are minimal and maximal solutions to the Bellman equation~\eqref{eq:Bellman}, respectively. For bounded $G$ one may show that $\underline{w}\equiv \overline{w}$; see Proposition 5 and Corollary 6 in~\cite{JelPitSte2019a} for details. However, for unbounded $G$ this may no longer be true; see Example~\ref{ex:3}. Thus, it is interesting to characterise the structure of solutions to~\eqref{eq:Bellman}. We start with the following lemma giving a martingale characterisation of solutions to the Bellman equation.

\begin{lemma}\label{lm:bellman1}
Let $v$ be a non-negative measurable solution to~\eqref{eq:Bellman} and let $\tau_v:=\inf\{n\in\mathbb{N}: v(X_n)\geq G(X_n)\}$. Define the process
\begin{equation}\label{eq:z_n}
z_v(n):=\exp\left(\sum_{i=0}^{n-1} g(X_i)+v(X_n)\right), \quad n\in \mathbb{N}.
\end{equation}
Then, for any stopping time $\tau$ we get that $(z_v(\tau \wedge n))$, $n\in \mathbb{N}$, is a submartingale. Also, $(z_v(\tau_v \wedge n))$, $n\in \mathbb{N}$, is a martingale.
\end{lemma}
\begin{proof}
First, using the inequality $e^{g(x)}\mathbb{E}_x\left[e^{v(X_1)}\right]\geq e^{v(x)}$, $x\in E$, and Markov property, for any $x\in E$ and $n\in \mathbb{N}$, we get
\[
\mathbb{E}_x\left[z_v(n+1)|\mathcal{F}_n\right]=e^{\sum_{i=0}^{n-1} g(X_i)} e^{g(X_n)}\mathbb{E}_{X_n}\left[e^{v(X_1)}\right]\geq z_v(n)
\]
and the process $(z_v(n))$, $n\in \mathbb{N}$, is a submartingale. Thus, using Doob optional stopping theorem, we get that for any stopping time $\tau$ the process $(z_v(\tau \wedge n))$, $n\in \mathbb{N}$, is also a submartingale.

Second, note that on the set $\{\tau_v>n\}$, we get $e^{v(X_n)}=e^{g(X_n)}\mathbb{E}_{X_n}\left[e^{v(X_1)}\right]$. Thus, for any $x\in E$ and $n\in \mathbb{N}$, we get
\begin{align*} 
\mathbb{E}_x\left[z_v(\tau_v\wedge (n+1))|\mathcal{F}_n \right] & = 1_{\{\tau_v\leq n\}} z_v(\tau_v) + 1_{\{\tau_v>n\}}e^{\sum_{i=0}^{n} g(X_i)}\bE_x\left[e^{v(X_{n+1})}|\mathcal{F}_n\right]  \\
& = 1_{\{\tau_v\leq n\}}  z_v(\tau_v) +1_{\{\tau_v>n\}}e^{\sum_{i=0}^{\tau_v\wedge n} g(X_i)}\mathbb{E}_{X_n}\left[e^{v(X_1)}\right]\\
& = 1_{\{\tau_v\leq n\}}  z_v(\tau_v\wedge n) +1_{\{\tau_v>n\}}e^{\sum_{i=0}^{\tau_v\wedge n-1} g(X_i)}e^{v(X_{\tau_v\wedge n})}\\
& = z_v(\tau_v\wedge n),
\end{align*}
which concludes the proof.
\end{proof}

Now we show that the minimal and maximal solutions to the Bellman equation~\eqref{eq:Bellman} coincide with the value functions given by~\eqref{eq:opt_stop_unb_disc} and~\eqref{eq:opt_stop_unb_disc2}.

\begin{theorem}\label{th:und_over_prob}
Let the maps $u$ and $w$ be given by~\eqref{eq:opt_stop_unb_disc} and~\eqref{eq:opt_stop_unb_disc2}, respectively. Then,
\begin{enumerate}
\item\label{p:th:und_over_prob:1} We get $ u\equiv\underline{w}$ and $w\equiv \overline{w}$, where the maps $\underline{w}$ and $\overline{w}$ are given by~\eqref{eq:und_over_w};
\item\label{p:th:und_over_prob:2} The functions $u$ and $w$ are solutions to~\eqref{eq:Bellman};
\item\label{p:th:und_over_prob:3} For any solution $v$ to the Bellman equation~\eqref{eq:Bellman} satisfying $0\leq v(\cdot)\leq G(\cdot)$ we get $u(\cdot)\leq v(\cdot)\leq w(\cdot)$.
\end{enumerate}
\end{theorem}

\begin{proof} Recalling~\eqref{eq:Bellman_bounds} and the successive discussion we get that~\eqref{p:th:und_over_prob:2} and~\eqref{p:th:und_over_prob:3} follow directly from~\eqref{p:th:und_over_prob:1}. Thus, it is enough to show~\eqref{p:th:und_over_prob:1}. For transparency, we split the rest of the proof into two parts: (1) proof of $u\equiv \underline{w}$; (2) proof of $w \equiv \overline{w}$.

\medskip
\noindent
\textit{Part 1.} We show that $u\equiv\underline{w} $. Recalling $\underline{w}_n$ from~\eqref{eq:underline_w} and Proposition~\ref{pr:disc_und_over}, for any $n\in \mathbb{N}$ and $x\in E$, we get
\begin{align*}
e^{\underline{w}_n(x)}&=\inf_{\tau\in \mathcal{T}_0}\mathbb{E}_x\left[e^{\sum_{i=0}^{\tau\wedge n-1} g(X_i)+1_{\{ \tau<n\}}G(X_{\tau})}\right]\\
& \leq \inf_{\tau\in \mathcal{T}_0}\mathbb{E}_x\left[e^{\sum_{i=0}^{\tau-1} g(X_i)+G(X_{\tau})}\right]=e^{u(x)},
\end{align*}
where the inequality follows from non-negativity of $g$ and $G$. Letting $n\to\infty$ we get $\underline{w} \leq u $. Now, let us define 
\begin{align}
\underline{z}(n)&:=\exp\left(\sum_{i=0}^{n-1} g(X_i)+\underline{w}(X_n)\right), \quad n\in \mathbb{N};\label{eq:th:und_over_prob:mart}\\
\underline{\tau}& :=\inf\{n\in\mathbb{N}: \underline{w}(X_n)\geq G(X_n)\};
\end{align}
and note that by Lemma~\ref{lm:bellman1} the process $(\underline{z}(\underline{\tau}\wedge n))$, $n\in \mathbb{N}$, is a martingale. Also, recalling that $g(\cdot)\geq c>0$ and $\underline{w}(\cdot)\geq 0$, and using Fatou Lemma, for any $x\in E$, we get
\begin{align*}
\mathbb{E}_x\left[e^{\underline\tau c }\right] &= \mathbb{E}_x\left[\liminf_{n\to\infty}e^{(\underline\tau\wedge n) c }\right]  \\
& \leq \liminf_{n\to\infty}\mathbb{E}_x\left[e^{\sum_{i=0}^{\underline\tau\wedge n-1} g(X_i)+\underline{w}(X_{\underline\tau\wedge n})}\right]\\
& =\mathbb{E}_x\left[ \underline{z}(0)\right]=e^{\underline{w}(x)}\leq e^{G(x)}<\infty.
\end{align*}
In particular, we get $\mathbb{P}_x[\underline{\tau}<\infty]=1$. Thus, noting that $\underline{w}(X_{\underline\tau})=G(X_{\underline\tau})$, we get
\begin{align}\label{eq:pr:und_over_prob:1}
e^{u(x)}&\leq \mathbb{E}_x\left[e^{\sum_{i=0}^{\underline\tau-1} g(X_i)+G(X_{\underline\tau})}\right] \nonumber\\
&=\mathbb{E}_x\left[e^{\sum_{i=0}^{\underline\tau-1} g(X_i)+\underline{w}(X_{\underline\tau})}\right]\nonumber\\
&=\mathbb{E}_x\left[\liminf_{n\to\infty} e^{\sum_{i=0}^{\underline\tau\wedge n-1} g(X_i)+\underline{w}(X_{\underline\tau\wedge n})}\right]\nonumber\\
&\leq \liminf_{n\to\infty} \mathbb{E}_x\left[e^{\sum_{i=0}^{\underline\tau\wedge n-1} g(X_i)+\underline{w}(X_{\underline\tau\wedge n})}\right] =\mathbb{E}_x\left[ \underline{z}(0)\right]= e^{\underline{w}(x)},
\end{align}
hence $u\equiv \underline{w}$, which concludes the proof of this part.

\medskip
\noindent
\textit{Part 2.}
We show that $ w\equiv\overline{w}$. Recalling Proposition~\ref{pr:disc_und_over} and the maps $\overline{w}_k$ from~\eqref{eq:overline_w}, for any $k\in \mathbb{N}$ and $x\in E$, we get
\begin{align*}
e^{w(x)}&\leq \inf_{\tau\leq k}\liminf_{n\to\infty}\mathbb{E}_x\left[e^{\sum_{i=0}^{\tau\wedge n-1} g(X_i)+G(X_{\tau\wedge n})}\right]\\
& = \inf_{\tau\leq k}\mathbb{E}_x\left[e^{\sum_{i=0}^{\tau-1} g(X_i)+G(X_{\tau})}\right]=e^{\overline{w}_k(x)}.
\end{align*}
Thus, letting $k\to\infty$, we get $w\leq \overline{w}$. Also, for any $n\in \mathbb{N}$ and $\hat\tau\in \mathcal{T}_0$ we get
\[
\inf_{\tau \leq n}  \E_x\left[e^{\sum_{i=0}^{\tau-1} g(X_i)+G(X_\tau)}\right] \leq  \E_x\left[e^{\sum_{i=0}^{\hat\tau\wedge n-1} g(X_i)+G(X_{\hat\tau\wedge n})}\right].
\]
Thus, letting $n\to\infty$ and taking infimum over $\hat\tau\in \mathcal{T}_0$, we get $\overline{w}\leq w$, which concludes the proof.
\end{proof}

\begin{remark}
From Theorem~\ref{th:und_over_prob} we deduce that in the unbounded case the family of finite time horizon stopping problems may not converge to their infinite horizon version. More specifically, from Proposition~\ref{pr:disc_und_over} we get that the function $\overline{w}_n$ may be seen as a finite horizon counterpart of $u$, with stopping times bounded by $n\in \mathbb{N}$. Thus, one might conjecture that $\overline{w}_n$ converges to $u$ as $n\to\infty$. However, from Theorem~\ref{th:und_over_prob} we get $\overline{w}_n \to w$ as $n\to\infty$ and from Examples~\ref{ex:1} and~\ref{ex:3} we see that in general $u\neq w$. Also, note that Theorem~\ref{th:und_over_prob} provides a finite horizon approximation scheme for $u$; this can be done with the help of the family $\underline{w}_n$.
\end{remark}

From the proof of Theorem~\ref{th:und_over_prob} we get a useful corollary about the optimal stopping time for $u$.
\begin{corollary}\label{cor:optimal_u}
Let $u$ be given by~\eqref{eq:opt_stop_unb_disc}. Then, the stopping time 
\begin{equation}\label{eq:cor:optimal_u}
\underline{\tau}=\inf\{n\in\mathbb{N}: \underline{w}(X_n)\geq G(X_n)\}
\end{equation}
is optimal for $u$. Also, the process $(\underline{z}(n\wedge \underline{\tau}))$, $n\in \mathbb{N}$, with $\underline{z}$ given by~\eqref{eq:th:und_over_prob:mart}, is a uniformly integrable martingale. 
\end{corollary}
\begin{proof}
Optimality of $\underline{\tau}$ follows directly from~\eqref{eq:pr:und_over_prob:1}. Also, martingale property of $(\underline{z}(n\wedge \underline{\tau}))$, $n\in \mathbb{N}$, follows from Lemma~\ref{lm:bellman1}. Finally, uniform integrability follows from~\eqref{eq:pr:und_over_prob:1}.
\end{proof}

Now we formulate a sufficient condition for the identity $u\equiv w$.
To ease the notation, we define the process
\[
Z_n:=\exp\left(\sum_{i=0}^{ n-1}g(X_i)+G(X_{n})\right), \quad n\in \mathbb{N}.
\]

\begin{theorem}\label{th:suff_disc}
Let $u$ and $w$ be given by~\eqref{eq:opt_stop_unb_disc} and~\eqref{eq:opt_stop_unb_disc2}, respectively. Also, let $\underline{\tau}=\inf\{t\geq 0: \underline{w}(X_t)\geq G(X_t)\}$. If the process $(Z_{n\wedge \underline{\tau}})$, $n\geq 0$, is uniformly integrable, then we get $u\equiv w$.
\end{theorem}
\begin{proof}
Recall that by Corollary~\ref{cor:optimal_u} the stopping time $\underline{\tau}$ is optimal for $u$. Thus, using uniform integrability of $(Z_{n\wedge \underline{\tau}})$, $n\geq 0$, for any $x\in E$, we get
\begin{align*}
e^{w(x)}\leq \lim_{n\to\infty} \E_x\left[e^{\sum_{i=0}^{\underline\tau\wedge n-1} g(X_i)+G(X_{\underline\tau\wedge n})}\right]&=\E_x\left[e^{\sum_{i=0}^{\underline\tau-1} g(X_i)+G(X_{\underline\tau})}\right]=e^{u(x)}.
\end{align*}
Recalling that we always get $u\leq w$, we conclude the proof.
\end{proof}

\begin{remark}\label{rm:lm:suff_disc:2}
By analogy to~\eqref{eq:cor:optimal_u}, let us define $\overline{\tau}:=\inf\{t\geq 0: \overline{w}(X_t)\geq G(X_t)\}$. Since $\underline{w}\leq \overline{w}$, we get $\overline{\tau}\leq \underline{\tau}$, where is given by~\eqref{eq:cor:optimal_u}. Based on the condition from Theorem~\ref{th:suff_disc} it is natural to ask whether uniform integrability of $(Z_{\overline{\tau}\wedge n})$ is also sufficient for $u\equiv w$. However, as discussed in Remark~\ref{rm:ex3}, this is not the case.
\end{remark}

\section{Continuous time optimal stopping}\label{S:continuous}
In this section, by analogy to~\eqref{eq:opt_stop_unb_disc} and~\eqref{eq:opt_stop_unb_disc2}, we consider the continuous time optimal stopping problems
\begin{align}
u(x)&:=\inf_{\tau}\ln\mathbb{E}_x\left[\exp\left(\int_{0}^{\tau} g(X_s)ds +G(X_{\tau})\right)\right], \quad x\in E;\label{eq:opt_stop_unb_cont}\\
w(x)&:=\inf_{\tau}\liminf_{T\to\infty}\ln\mathbb{E}_x\left[\exp\left(\int_{0}^{\tau\wedge T} g(X_s)ds+G(X_{\tau\wedge T})\right)\right], \quad x\in E.\label{eq:opt_stop_unb_cont2}
\end{align}
Assuming~\eqref{cost_functions}--\eqref{Feller}, we prove several regularity properties of the maps $u$ and $w$. Also, we show various approximation results, including finite time horizon limits. These results extend the analysis from~\cite{JelPitSte2019a} to the case when $G$ is unbounded from above.

First, by analogy to Proposition~\ref{pr:disc_und_over} we consider the finite time horizon optimal stopping problems. For any $T\geq 0$, let us define
\begin{align}
\underline{w}_T(x) & :=\inf_{\tau\leq T} \ln\E_x\left[e^{\int_0^\tau g(X_s)ds + 1_{\{\tau<T\}}G(X_\tau) }\right],\quad x\in E,\label{eq:w^t_approx_from_below}\\
\overline{w}_T(x) & :=\inf_{\tau\leq T} \ln \E_x\left[e^{\int_0^{\tau} g(X_s)ds+G(X_{\tau})}\right], \quad  x\in E.\label{eq:w^T_approx_from_above}
\end{align}
In Proposition~\ref{pr:finite_hor} we summarise the properties of the maps $(T,x)\mapsto \underline{w}_T(x)$ and $(T,x)\mapsto \overline{w}_T(x)$.  
The proof is deferred to Appendix~\ref{S:proofs}.

\begin{proposition}\label{pr:finite_hor}
Let the maps $(\underline w_T)$ and $(\overline w_T)$ be given by~\eqref{eq:w^t_approx_from_below} and~\eqref{eq:w^T_approx_from_above}, respectively. Then,
\begin{enumerate}
\item The map $(T,x)\mapsto \underline{w}_T(x)$ is jointly continuous and, for any $x\in E$, the map $T\mapsto \underline{w}_T(x)$ is non-decreasing. Also, for any $T\geq 0$, an optimal stopping time for $\underline w_T$ is given by
\begin{equation}\label{optstop}
\underline{\tau}_T:=\inf\left\{t\geq 0: \underline{w}_{T-t}(X_t)=G(X_t)\right\}\wedge T.
\end{equation}
Moreover, the process
\[
\underline{z}_T(t):=e^{\int_0^{t\wedge T} g(X_s) ds+\underline{w}_{T-t\wedge T}(X_{t\wedge T})}, \quad t\geq 0,
\]
is a submartingale and $(\underline{z}_T(t\wedge \underline{\tau}_T))$, $t\geq 0$, is a martingale.
\item The map $(T,x)\mapsto \overline{w}_T(x)$ is jointly continuous and, for any $x\in E$, the map $T\mapsto \overline{w}_T(x)$ is non-increasing. Also, for any $T\geq 0$, an optimal stopping time for $\overline w_T$ is given by
\begin{equation}\label{eq:prop:coninuity_w_T:optstop}
\overline{\tau}_T:=\inf\left\{t\geq 0:\overline{w}_{T-t}(X_t)=G(X_t)\right\}.
\end{equation}
Moreover, the process
\[
\overline{z}_T(t):=e^{\int_0^{t\wedge T} g(X_s) ds+\overline{w}_{T-t\wedge T}(X_{t\wedge T})}, \quad t\geq 0,
\]
is a submartingale and $(\overline{z}_T(t\wedge \overline{\tau}_T))$, $t\geq 0$, is a martingale.
\end{enumerate}
\end{proposition}

Based on Proposition~\ref{pr:finite_hor} we may define the limits
\begin{equation}\label{eq:w.both2}
\underline w(x):=\lim_{T\to\infty}\underline w_{T}(x)\quad\textrm{and}\quad\overline w(x) :=\lim_{T\to\infty}\overline w_{T}(x), \quad x\in E.
\end{equation}
Let us now link the functions $\underline{w}$ and $\overline{w}$ with~\eqref{eq:opt_stop_unb_cont} and~\eqref{eq:opt_stop_unb_cont2}.

\begin{theorem}\label{th:equiv_cont}
Let the functions $u$ and $w$ be given by~\eqref{eq:opt_stop_unb_cont} and~\eqref{eq:opt_stop_unb_cont2}, respectively. Also, let $\underline{w}$ and $\overline{w}$ be given by~\eqref{eq:w.both2}. Then we get $u\equiv \underline{w}$ and $w\equiv \overline{w}$. Also, $u$ is lower semicontinuous and $w$ is upper semicontinuous.
\end{theorem}
\begin{proof}
The proof for $w\equiv \overline{w}$ follows the lines of the second step in the proof of Theorem~\ref{th:und_over_prob} and is omitted for brevity. Now we show that $u\equiv \underline{w}$. The proof is partially based on Theorem 15 in~\cite{JelPitSte2019a}. For transparency, we present it in detail.

First, recalling non-negativity of $g$ and $G$, for any $T\geq 0$ and $x\in E$ we get
\begin{align*}
e^{\underline{w}_T(x)}&=\inf_{\tau}\mathbb{E}_x\left[e^{\int_0^{\tau\wedge T} g(X_s)ds+1_{\{ \tau<T\}}G(X_{\tau})}\right]\\
& \leq \inf_{\tau}\mathbb{E}_x\left[e^{\int_{0}^{\tau} g(X_s)ds+G(X_{\tau})}\right]=e^{u(x)}.
\end{align*}
Thus, letting $T\to\infty$, we get $\underline{w}\leq u$. Let us now show the reverse inequality.

For any $T>0$, let $\underline{\tau}_T$ be an optimal stopping time for $\underline{w}_T$, given by the formula~\eqref{optstop}. Define 
\begin{equation}\label{eq:limit_tau}
{\hat{\underline\tau}}_T:=\inf\left\{t\geq 0: \underline{w}_{T-t}(X_t)\geq G(X_t)\right\}
\end{equation}
and observe that $\underline{\tau}_T={\hat{\underline\tau}}_T\wedge T$. By monotonicity of the sequence $(\underline{w}_n(x))_{n\in \mathbb{N}}$, we get ${\hat{\underline\tau}}_{n+1}\leq {\hat{\underline\tau}}_n$. Thus,  for any $n\in\bN$, on the set $\{{{\underline\tau}}_n<n\}$, we get ${\hat{\underline\tau}}_{n}={{\underline\tau}}_{n}$, thus ${\hat{\underline\tau}}_{n+1}={{\underline\tau}}_{n+1}$, and consequently ${{\underline\tau}}_{n+1}\leq {{\underline\tau}}_n$. Moreover, recalling that $g(\cdot)\geq c>0$ and $G(\cdot)\geq 0$, for any $x\in E$, we get
\begin{equation*}
e^{G(x)}\geq e^{\underline{w}_{T}(x)}=\E_x\left[e^{\int_0^{\underline{\tau}_T} g(X_s)ds + 1_{\{\underline{\tau}_T<T\}}G(X_{\underline{\tau}_T})}\right]\geq \E_x\left[1_{\{\underline{\tau}_T=T\}}\right]e^{cT}.
\end{equation*}
Consequently, for any $x\in E$, we get
$
\sum_{n=1}^\infty \bP_x\left[\underline{\tau}_{n}=n\right]\leq \sum_{n=1}^\infty{\frac{e^{G(x)}}{e^{cn}}}<\infty$. Hence, by Borel-Cantelli Lemma, for any $x\in E$, we get $\mathbb{P}_x\left[ \bigcup_{n=1}^\infty \{\underline\tau_n<n\}\right]=1$, and consequently the stopping time
\begin{equation}\label{eq:hat_tau}
\hat{\tau}:=\lim_{n\to \infty} \underline{\tau}_n
\end{equation}
is well defined. Also, we get that $\mathbb{P}_x[\hat\tau<\infty]=1$, $x\in E$. This follows from the fact that for $\mathbb{P}_x$ almost all $\omega \in \Omega$, starting from some $n$ (depending on $\omega)$, the sequence $(\underline{\tau}_n(\omega))$ is non-increasing. Thus, using right continuity of $(X_t)$ and Fatou Lemma, for any $x\in E$, we get
\begin{align}\label{eq:th:w_continuity:ineq2}
e^{u(x)} \leq \E_x\left[e^{\int_0^{\hat{\tau}} g(X_s)ds+G(X_{\hat{\tau}})}\right] & = \E_x\left[\lim_{n\to \infty}\left(e^{\int_0^{\underline{\tau}_n} g(X_s)ds+1_{\{\underline{\tau}_n<n\}}G(X_{\underline{\tau}_n})}\right)\right] \nonumber \\
& \leq \liminf_{n\to\infty} e^{\underline{w}_n(x)} = e^{\underline{w}(x)},
\end{align}
which concludes the proof of $u\equiv \underline{w}$.

Finally, recalling that by Proposition~\ref{pr:finite_hor} the map $u$ is an increasing limit of continuous functions, we get that $u$ is lower semicontinuous. Using similar argument for $w$ we get upper semicontinuity.
\end{proof}

\begin{remark}\label{rm:opt_tau_cont}
From the proof of Theorem~\ref{th:equiv_cont} we get that the stopping time $\hat{\tau}$ given by~\eqref{eq:hat_tau} is optimal for $u\equiv \underline{w}$; see~\eqref{eq:th:w_continuity:ineq2}. Also, note that in the proof we showed that $\mathbb{P}_x[\hat\tau<\infty]=1$, $x\in E$; see the discussion following~\eqref{eq:hat_tau}.
\end{remark}

In Theorem~\ref{th:equiv_cont} we showed that the function $u$ given by~\eqref{eq:opt_stop_unb2} may be seen as a limit of finite horizon stopping problems $\underline{w}_T$. Let us now show that $u$ may also be approximated by stopping problems with truncated terminal cost function. More explicitly, for any $n\in \mathbb{N}$, we define
\begin{equation}\label{eq:opt_stop_bound}
u_n(x):=\inf_{\tau}\ln\mathbb{E}_x\left[\exp\left(\int_0^\tau g(X_s)ds+G(X_\tau)\wedge n\right)\right],\quad  x \in E.
\end{equation}
Clearly, we have $u_n(x)\leq u_{n+1}(x)\leq u(x)$ for any $x\in E$ and $n\in \mathbb{N}$. 
In Theorem~\ref{th:inf1} we link the functions $u$ and $u_n$. 
\begin{theorem}\label{th:inf1}
Let the functions $u$ and $u_n$ be given by~\eqref{eq:opt_stop_unb_cont} and~\eqref{eq:opt_stop_bound}, respectively. Then, for any $x\in E$, we get $u(x)=\lim_{n\to\infty} u_n(x)$. 
\end{theorem}
\begin{proof}
Let us define the sequence of events $A_n:=\{G(X_{\tau_n})\leq n\}, n\in \mathbb{N}$, where 
\[
\tau_n:=\inf\{t\geq 0:u_n(X_t)\geq G(X_t)\wedge n\}.
\]
Using Theorem 15 from~\cite{JelPitSte2019a} we get that $\tau_n$ is an optimal stopping time for $u_n(x)$, $x\in E$, $n\in \mathbb{N}$. Also, recalling that $g(\cdot)\geq 0$, for any $x\in E$ and $k\in \mathbb{N}$, we get
\begin{align*}
e^{G(x)}\geq e^{u_k(x)}&=\mathbb{E}_x\left[e^{\int_0^{\tau_k} g(X_s)ds+G(X_{\tau_k})\wedge k}\right]\\
&\geq \mathbb{E}_x\left[1_{A_k^c}e^{\int_0^{\tau_k} g(X_s)ds+G(X_{\tau_k})\wedge k}\right]\geq \mathbb{P}_x\left[A_k^c\right]e^k.
\end{align*}
Thus $\mathbb{P}_x\left[A_k^c\right]\leq \frac{e^{G(x)}}{e^k}$ and $\sum_{k=1}^\infty \mathbb{P}_x\left[A_k^c\right]<\infty$. Hence, from Borel-Cantelli Lemma, for any $x\in E$, we get
\begin{equation}\label{eq:calc1}
\mathbb{P}_x\left[ \cup_{n=1}^\infty \cap_{k=n}^\infty A_k\right]=1.
\end{equation}
Let us fix $n\in \mathbb{N}$ and note that on the set $\cap_{k=n}^\infty A_k$, for any $j\geq 0$, we get
\begin{multline*}
u_{n+j+1}(X_{\tau_{n+j}})\geq u_{n+j}(X_{\tau_{n+j}})\geq G(X_{\tau_{n+j}})\wedge(n+j)\\
=G(X_{\tau_{n+j}})\geq G(X_{\tau_{n+j}})\wedge (n+j+1).
\end{multline*}
Thus, on the set $\cap_{k=n}^\infty A_k$, for any $j\geq 0$, we get $\tau_{n+j+1} \leq \tau_{n+j}$. Combining this with~\eqref{eq:calc1}, we may define the stopping time $\tilde{\tau}:=\lim_{n\to\infty}\tau_n$. 
Moreover, we get that $\tilde\tau$ is almost surely finite since, for any $n\in \mathbb{N}$, the stopping time $\tau_n$ is almost surely finite; see Remark 16 in~\cite{JelPitSte2019a} for details. Thus, using right continuity of $X$ and Fatou Lemma, for any $x\in E$, we get
\begin{align*}
e^{u(x)}&\leq \mathbb{E}_x\left[e^{\int_0^{\tilde\tau} g(X_s)ds+G(X_{\tilde\tau})}\right]\nonumber\\
&=\mathbb{E}_x\left[\lim_{n\to\infty}e^{\int_0^{\tau_n } g(X_s)ds+G(X_{\tau_n})\wedge n}\right]\nonumber\\
& \leq \lim_{n\to\infty}\mathbb{E}_x\left[e^{\int_0^{\tau_n } g(X_s)ds+G(X_{\tau_n })\wedge n}\right]=\lim_{n\to\infty}e^{u_n(x)}\leq e^{u(x)}.
\end{align*}
and consequently $\lim_{n\to\infty} u_n(x)=u(x)$, $x\in E$.
\end{proof}
\begin{remark}
By analogy to Theorem~\ref{th:inf1} one could try to approximate the function $w$ from~\eqref{eq:opt_stop_unb_cont2} by the family
\begin{equation*}\label{eq:opt_stop_bound2}
w_n(x):=\inf_{\tau}\liminf_{T\to\infty}\ln\mathbb{E}_x\left[\exp\left(\int_0^{\tau\wedge T} g(X_s)ds+G(X_{\tau\wedge T})\wedge n\right)\right],\,  n\in \mathbb{N}, \, x \in E.
\end{equation*}
However, since for any $n\in \mathbb{N}$ the map $G(\cdot) \wedge n$ is bounded, using Theorem 15 from~\cite{JelPitSte2019a} we get $u_n\equiv w_n$ and by Theorem~\ref{th:inf1} we get $w_n\to u$. In fact, the identity $u_n\equiv w_n$  may also be deduced from Corollary~\ref{cor:identity_bounded} in this paper.
\end{remark}

\section{Continuous time Bellman equation}\label{S:Bellman_cont}

In this section we extend the results from Section~\ref{S:discrete} to the continuous time case. We consider the continuous time Bellman equation which takes the form of optimal stopping dynamic programming principle
\begin{equation}\label{eq:Bellman_cont}
e^{v(x)}=\inf_{\tau} \mathbb{E}_x\left[e^{\int_0^{\tau \wedge t} g(X_s) ds+1_{\{\tau< t \}}G(X_\tau)+1_{\{\tau\geq  t \}}v(X_t)}\right], \quad t\geq 0,\,x\in E.
\end{equation}
By analogy to Section~\ref{S:discrete} we show that the maps $u$ and $w$ are minimal and maximal solutions to this equation, respectively.

First, note that~\eqref{eq:Bellman_cont} may be expressed in the operator form as
\[
\Phi_t v(x)=v(x), \quad t\geq 0,\, x\in E,
\]
where, for any $t\geq 0$, the operator $\Phi_t$ is given by
\begin{equation}\label{eq:Phi}
\Phi_t h(x):=\inf_{\tau} \ln \mathbb{E}_x\left[e^{\int_0^{\tau \wedge t} g(X_s) ds+1_{\{\tau< t \}}G(X_\tau)+1_{\{\tau\geq t \}}h(X_t)}\right], \, x\in E,
\end{equation}
and $h:E\mapsto \mathbb{R}_+$ is a non-negative measurable function. To characterise the solutions to~\eqref{eq:Bellman_cont}, for any $t\geq 0$, let us define recursively
\begin{align}
\underline{v}_0^t(x) & =0, &  \underline{v}_{n+1}^t(x) & =\Phi_t \underline{v}_n^t(x), & \quad n\in \mathbb{N},\, x\in E;\label{eq:underline_w_cont}\\
\overline{v}_0^t(x) & =G(x), &  \overline{v}_{n+1}^t(x) & =\Phi_t \overline{v}_n^t(x), & \quad n\in \mathbb{N},\, x\in E.\label{eq:overline_w_cont}
\end{align}
We start with linking $\underline{v}_n^t$ and $\overline{v}_n^t$ with the functions $\underline{w}_T$ and $\overline{w}_T$.
\begin{proposition}\label{pr:hat_w}
For any $t\geq 0$ and $n\in \mathbb{N}$, let the maps $\underline{v}_n^t$ and $\overline{v}_n^t$ be given by~\eqref{eq:underline_w_cont} and~\eqref{eq:overline_w_cont}, respectively. Then, 
\begin{enumerate}
\item\label{p:pr:hat_w:1} For any $t\geq 0$ and $n\in \mathbb{N}$, we get $\underline{v}_n^t\equiv \underline{w}_{nt}$ and $\overline{v}_n^t\equiv \overline{w}_{nt}$, where the functions $\underline{w}_T$ and $\overline{w}_T$ are given by~\eqref{eq:w^t_approx_from_below} and~\eqref{eq:w^T_approx_from_above}, respectively.
\item\label{p:pr:hat_w:2} For any $x\in E$ and $t\geq 0$, we get 
\[
\lim_{n\to\infty} \underline{v}_n^t(x)= u(x)\quad \text{and} \quad \lim_{n\to\infty} \overline{v}_n^t(x)=w(x),
\]
 where the functions $u$ and $w$ be given by~\eqref{eq:opt_stop_unb_cont} and~\eqref{eq:opt_stop_unb_cont2}, respectively. In particular, the limits $\lim_{n\to\infty} \underline{v}_n^t(x)$ and $\lim_{n\to\infty} \overline{v}_n^t(x)$ are well-defined and independent of $t\geq 0$.
\end{enumerate}
\end{proposition}
\begin{proof}
For transparency, we split the proof into two parts.

\medskip
\noindent
\textit{Proof of \eqref{p:pr:hat_w:1}.} 
We present the proof only for $\underline{v}_n^t$; the argument for $\overline{v}_n^t$ is similar and is omitted for brevity. Also, for the notational convenience we set $t=1$; the general case follows the same logic. 

We proceed by induction. The claim for $n=0$ follows directly from the definition. Let us assume that for some $n\in \mathbb{N}$ we get $\underline{v}_n^1\equiv \underline{w}_{n}$. Define the process $\underline{z}_{n+1}(t):=e^{\int_0^{t\wedge (n+1)} g(X_s) ds+\underline{w}_{n+1-t\wedge (n+1)}(X_{t\wedge (n+1)})}$, $t\geq 0$. Using Proposition~\ref{pr:finite_hor} and Doob optional stopping theorem, for any stopping time $\tau$ we get that the process $(\underline{z}_{n+1}(\tau\wedge t))$, $t\geq 0$, is a submartingale. In particular, for any $x\in E$, we get $\mathbb{E}_x[\underline{z}_{n+1}(0)]\leq \inf_{\tau} \mathbb{E}_x\left[\underline{z}_{n+1}(\tau \wedge 1)\right]$. Then, recalling that $\underline{w}_T(x)\leq G(x)$ for any $x\in E$ and $T\geq 0$, we get
\begin{align}\label{eq:lm:hat_w:1}
e^{\underline{w}_{n+1}(x)}=\mathbb{E}_x[\underline{z}_{n+1}(0)]&\leq \inf_{\tau} \mathbb{E}\left[e^{\int_0^{\tau \wedge 1} g(X_s) ds+\underline{w}_{n+1-\tau\wedge 1}(X_{\tau\wedge 1})}\right]\nonumber\\
& \leq  \inf_{\tau} \mathbb{E}\left[e^{\int_0^{\tau \wedge 1} g(X_s) ds+1_{\{\tau< 1 \}}G(X_\tau)+1_{\{\tau\geq 1 \}}\underline{w}_{n}(X_{1})}\right].
\end{align}
Recall that by Proposition~\ref{pr:finite_hor} the process $(\underline{z}_{n+1}( \underline{\tau}_{n+1}\wedge t))$, $t\geq 0$, is a martingale, where $\underline{\tau}_{n+1}:=\inf\{t\geq 0: \underline{w}_{n+1-t}(X_t)= G(X_t)\}\wedge (n+1)$. Also, on the event $\{\underline{\tau}_{n+1}<n+1\}$ we get $\underline{w}_{n+1-\underline{\tau}_{n+1}}(X_{\underline{\tau}_{n+1}})=G(X_{\underline{\tau}_{n+1}})$. Thus, for any $x\in E$, we get
\[
e^{\underline{w}_{n+1}(x)}= \mathbb{E}\left[e^{\int_0^{\underline{\tau}_{n+1} \wedge 1} g(X_s) ds+1_{\{\underline{\tau}_{n+1}< 1 \}}G(X_{\underline{\tau}_{n+1}})+1_{\{\underline{\tau}_{n+1}\geq  1 \}}\underline{w}_{n}(X_{1})}\right].
\]
Combining this with~\eqref{eq:lm:hat_w:1} and using induction assumption, for any $x\in E$, we get
\begin{align*}
e^{\underline{w}_{n+1}(x)}&=\inf_{\tau} \mathbb{E}\left[e^{\int_0^{\tau \wedge 1} g(X_s) ds+1_{\{\tau< 1 \}}G(X_\tau)+1_{\{\tau\geq  1 \}}\underline{w}_{n}(X_{1})}\right]\\
&=e^{\Phi_1 \underline{w}_{n}(x)}= e^{\Phi_1 \underline{v}_{n}^1(x)}=e^{\underline{v}_{n+1}^1(x)},
\end{align*}
which concludes the proof of this point.

\medskip
\noindent
\textit{Proof of \eqref{p:pr:hat_w:2}.} Recalling~\eqref{p:pr:hat_w:1} and Theorem~\ref{th:equiv_cont}, for any $x\in E$ and $t\geq 0$, we get
\[
\lim_{n\to\infty} \underline{v}_n^t(x) = \lim_{n\to\infty} \underline{w}_{nt} (x) = \underline{w}(x)= u(x).
\]
Using similar argument we get $\lim_{n\to\infty} \underline{v}_n^t(x)=w(x)$, $t\geq 0$, $x\in E$, which concludes the proof.
\end{proof}

In the following theorem we characterise the solutions to the Bellman equation~\eqref{eq:Bellman_cont}. In particular, we get that $u$ and $w$ are minimal and maximal solutions to~\eqref{eq:Bellman_cont}, respectively. This may be seen as a continuous time version of Theorem~\ref{th:und_over_prob}.

\begin{theorem}\label{th:Bellman_cont}
Let the functions $u$ and $w$ be given by~\eqref{eq:opt_stop_unb_cont} and~\eqref{eq:opt_stop_unb_cont2}, respectively. Then,
\begin{enumerate}
\item\label{p:2} The functions $u$ and $w$ are solutions to~\eqref{eq:Bellman_cont}.
\item\label{p:3} For any solution $v$ to the Bellman equation~\eqref{eq:Bellman_cont} satisfying $0\leq v(\cdot)\leq G(\cdot)$ we get $u(\cdot)\leq v(\cdot)\leq w(\cdot)$.
\end{enumerate}
\end{theorem}
\begin{proof} For transparency, we split the proof into two parts.

\medskip
\noindent
\textit{Proof of \eqref{p:2}.} First, we prove that $u$ satisfies~\eqref{eq:Bellman_cont}. Let us define the process
\begin{equation}\label{eq:und_z}
\underline{z}(t):=e^{\int_0^t g(X_s) ds+ \underline{w}(X_t)}, \quad t\geq 0,
\end{equation}
where $\underline{w}$ is given by~\eqref{eq:w.both2}. We show that $(\underline{z}(t))$, $t\geq 0$, is a submartingale. From Proposition~\ref{pr:finite_hor}, using submartingale property of $\underline{z}_T$, for any $T,t, h\geq 0$ and $x\in E$, we get
\begin{equation*}\label{eq:lm:mart_inf_hor:subm}
e^{\int_0^{t\wedge T} g(X_s) ds+ \underline{w}_{T-t\wedge T}(X_{t\wedge T})}\leq \mathbb{E}_x\left[e^{\int_0^{(t+h)\wedge T} g(X_s) ds+ \underline{w}_{T-(t+h)\wedge T}(X_{(t+h)\wedge T})}|\mathcal{F}_t\right].
\end{equation*}
Thus, recalling monotonicity of $T\mapsto\underline{w}_T(x)$, $x\in E$, and letting $T\to\infty$, for any $t, h\geq 0$ and $x\in E$, we get
\begin{equation}\label{eq:pr:Bellman_cont:0.5}
\underline{z}(t)= e^{\int_0^{t} g(X_s) ds+ \underline{w}(X_{t})}\leq \mathbb{E}_x\left[e^{\int_0^{t+h} g(X_s) ds+ \underline{w}(X_{t+h})}|\mathcal{F}_t\right]=\mathbb{E}_x\left[\underline{z}(t+h)|\mathcal{F}_t\right],
\end{equation}
which concludes the proof of submartingale property of $(\underline{z}(t))$, $t\geq 0$.

Next, using submartingale property of $(\underline{z}(t))$, $t\geq 0$, Doob optional stopping theorem, and the fact that $\underline{w}\leq G$, for any $t\geq 0$ and $x\in E$, we get
\begin{align}\label{eq:pr:Bellman_cont:1}
e^{\underline{w}(x)}=\mathbb{E}_x\left[\underline{z}(0)\right]&\leq \inf_{\tau} \mathbb{E}_x\left[\underline{z}(\tau \wedge t)\right]\nonumber\\
& \leq \inf_{\tau} \mathbb{E}_x\left[ e^{\int_0^{\tau \wedge t} g(X_s) ds+ 1_{\{\tau<t\}}G(X_{\tau})+1_{\{\tau\geq t\}}\underline{w}(X_t)}\right].
\end{align}
To conclude the proof we show that for any $t\geq 0$ and $x\in E$, we get
\begin{align}\label{eq:pr:Bellman_cont:2}
e^{\underline{w}(x)}=\mathbb{E}_x\left[e^{\int_0^{\hat{\tau}\wedge t} g(X_s) ds+ 1_{\{\hat{\tau}<t \}} G(X_{\hat{\tau}})+1_{\{\hat{\tau}\geq t \}}\underline{w}(X_t) }\right],
\end{align}
where the stopping time $\hat{\tau}$ is given by~\eqref{eq:hat_tau}. From Proposition~\ref{pr:finite_hor}, using martingale property of $(\underline{z}_T(t\wedge \underline{\tau}_T))$, for any $t\geq 0$, $T\geq t$,  and $x\in E$, we get
\begin{align*}
e^{\underline{w}_T(x)}=\mathbb{E}_x[\underline{z}_T(0)]&=\mathbb{E}_x\left[e^{\int_0^{\underline{\tau}_T\wedge t} g(X_s) ds+ \underline{w}_{T-\underline{\tau}_T\wedge t}(X_{\underline{\tau}_T\wedge t})}\right]\\
& = \mathbb{E}_x\left[e^{\int_0^{\underline{\tau}_T\wedge t} g(X_s) ds+ 1_{\{\underline{\tau}_T<t \}} G(X_{\underline{\tau}_T})+1_{\{\underline{\tau}_T\geq t \}}\underline{w}_{T-t}(X_t) }\right]
\end{align*}
Thus, using right-continuity of $X$, recalling Assumption~\eqref{IF}, and letting $T\to\infty$, we get~\eqref{eq:pr:Bellman_cont:2}. Combining this with~\eqref{eq:pr:Bellman_cont:1}, for any $t\geq 0$ and $x\in E$, we get
\[
e^{\underline{w}(x)}=\inf_{\tau} \mathbb{E}_x\left[ e^{\int_0^{\tau \wedge t} g(X_s) ds+ 1_{\{\tau<t\}}G(X_{\tau})+1_{\{\tau\geq t\}}\underline{w}(X_t)}\right].
\]
Recalling that by Theorem~\ref{th:equiv_cont} we get $u\equiv \underline{w}$, we conclude the proof that $u$ satisfies~\eqref{eq:Bellman_cont}.

Second, we prove that $w$ is also a solution to~\eqref{eq:Bellman_cont}. Noting that for any $x\in E$ the sequence $(\overline{v}_n^t(x))_{n\in \mathbb{N}}$ is non-increasing, using Proposition~\ref{pr:hat_w} and monotone convergence theorem, for any $t\geq 0$ and $x\in E$, we get
\begin{align*}
e^{w(x)}&=\inf_{n\in \mathbb{N}}e^{\overline{v}_{n+1}^t(x)}\\
&=\inf_{\tau}\inf_{n\in \mathbb{N}}\mathbb{E}_x\left[e^{\int_0^{\tau \wedge t} g(X_s) ds+1_{\{\tau< t \}}G(X_\tau)+1_{\{\tau\geq t\}}\overline{v}_{n}^t(X_1)}\right]\\
&=\inf_{\tau} \mathbb{E}_x\left[e^{\int_0^{\tau \wedge t} g(X_s) ds+1_{\{\tau< t \}}G(X_\tau)+1_{\{\tau\geq  t \}}w(X_t)}\right]=e^{\Phi_t w(x)},
\end{align*}
thus $w$ is a solution to~\eqref{eq:Bellman_cont}. 

\medskip
\noindent
\textit{Proof of \eqref{p:3}.} Recall that if $v$ is a solution to~\eqref{eq:Bellman_cont}, then $\Phi_t v= v$, for any $t\geq 0$.  Thus, recalling~\eqref{eq:underline_w_cont} and~\eqref{eq:overline_w_cont}, inductively we get $\underline{v}_n^t(x)\leq v(x)\leq \overline{v}_n^t(x)$ for any $t\geq 0$, $n\in \mathbb{N}$, and $x\in E$. Hence, letting $n\to\infty$ and using Proposition~\ref{pr:hat_w} we get \eqref{p:3}. 
\end{proof}

\begin{remark}
It should be noted that combining Theorem~\ref{th:equiv_cont} with Theorem~\ref{th:Bellman_cont} we get a possible numerical approximation scheme for extremal solutions to the Bellman equation. More specifically, we get that the map $u$, which is the smallest solution to~\eqref{eq:Bellman_cont}, could be approximated by finite horizon optimal stopping value functions $\underline{w}_T$ as $T\to\infty$. Also, note that in the Step 2 of the proof of Proposition~\ref{pr:finite_hor} we discuss a possible iterative procedure to approximate $\underline{w}_T$. Similar relations hold for the map $w$, which could be approximated by $\overline{w}_T$ as $T\to\infty$. 
\end{remark}

Based on Theorem~\ref{th:Bellman_cont} we get the following corollary.
\begin{corollary}\label{pr:equality_cont}
Let $u$ and $w$ be given by~\eqref{eq:opt_stop_unb_cont} and~\eqref{eq:opt_stop_unb_cont2}, respectively. Then, the following are equivalent
\begin{enumerate}
\item We get $u\equiv w$;
\item There is a unique solution to the Bellman equation~\eqref{eq:Bellman_cont} in the class of measurable functions $v$ satisfying $0\leq v(\cdot)\leq G(\cdot)$.
\end{enumerate}
\end{corollary}

In the next proposition we study the properties of continuous solutions to the Bellman equation~\eqref{eq:Bellman_cont}. This may be seen as a continuous time analogue of Lemma~\ref{lm:bellman1}. Note that, in contrast to the discrete time case, here we additionally require continuity of $v$. 
\begin{proposition}\label{pr:Bellman_cont_opt}
Let $v$ be a continuous solution to~\eqref{eq:Bellman_cont} satisfying $0\leq v(\cdot)\leq G(\cdot)$. Also, let us define
\[
\tau_v:=\inf\{ t\geq 0: v(X_t)\geq G(X_t)\}.
\]
Then, the infimum in~\eqref{eq:Bellman_cont} is attained for the stopping time $\tau_v$, i.e. for any $x\in E$ and $T\geq 0$ we get
\begin{equation}\label{eq:Bellman_cont_opt}
e^{v(x)}=\mathbb{E}_x\left[\exp\left(\int_0^{\tau_v \wedge T} g(X_s) ds+1_{\{\tau_v< T \}}G(X_{\tau_v})+1_{\{\tau_v\geq  T \}}v(X_T)\right)\right].
\end{equation}
Moreover, the process
\begin{equation}\label{eq:z_v}
z_v(t):=\exp\left(\int_0^{t}g(X_s)ds+v(X_{ t})\right), \quad t\geq 0,
\end{equation}
is a submartingale and $z_v(t\wedge \tau_v)$, $t\geq 0$, is a martingale.
\end{proposition}
\begin{proof}
For any $T\geq 0$ let us define
\[
e^{v_T(x)}=\inf_{\tau\leq T}\mathbb{E}_x\left[e^{\int_0^{\tau} g(X_s) ds+1_{\{\tau< T \}}G(X_\tau)+1_{\{\tau=T \}}v(X_T)}\right]
\]
and note that by~\eqref{eq:Bellman_cont} in fact we have $v_T\equiv v$ for any $T\geq 0$. In particular, we get that the map $(T,x)\mapsto v_T(x)$ is continuous. Hence, using Lemma~\ref{lm:opt_fin}, we get that the stopping time 
\begin{equation}\label{eq:Bellman_cont_optstop}
\tau_T:=\inf\{t\geq 0: v_{T-t}(X_t)\geq G(X_t)\}\wedge T=\tau_v\wedge T
\end{equation}
is optimal for $e^{v_T}$. Thus, for any $x\in E$ and $T\geq 0$, we get
\begin{align*}
e^{v(x)}=e^{v_T(x)}&=\mathbb{E}_x\left[e^{\int_0^{\tau_T} g(X_s) ds+1_{\{\tau_T< T \}}G(X_{\tau_T})+1_{\{\tau_T = T \}}v(X_T)}\right]\\
&=\mathbb{E}_x\left[e^{\int_0^{\tau_v \wedge T} g(X_s) ds+1_{\{\tau_v< T \}}G(X_{\tau_v})+1_{\{\tau_v \geq T \}}v(X_T)}\right]
\end{align*}
and~\eqref{eq:Bellman_cont_opt} holds. Finally, using Lemma~\ref{lm:opt_fin} again we also get the submartingale property of $z_v(t)$, $t\geq 0$, and the martingale property of $z_v(t\wedge \tau_v)$, $t\geq 0$. 
\end{proof}

In the following lemma we show when a continuous solution to the Bellman equation may be expressed as an expectation of the stopped value process. 

\begin{lemma}
Let $v$ be a continuous solution to~\eqref{eq:Bellman_cont} such that $0\leq v(\cdot)\leq G(\cdot)$. Also, let $\tau_v$ be as in Proposition~\ref{pr:Bellman_cont_opt}. Then, we get
\[
e^{v(x)}=\mathbb{E}_x\left[e^{\int_0^{\tau_v } g(X_s) ds+G(X_{\tau_v})}\right], \quad x\in E
\]
if and only if
\[
\lim_{T\to\infty}\mathbb{E}_x\left[1_{\{\tau_v\geq  T \}}e^{\int_0^{T} g(X_s) ds+v(X_T)}\right]=0, \quad x\in E.
\]
\end{lemma}
\begin{proof}
Let $v$ be a continuous solution to the Bellman equation~\ref{eq:Bellman_cont} satisfying $0\leq v(\cdot)\leq G(\cdot)$. Using Proposition~\ref{pr:Bellman_cont_opt}, for any $x\in E$ and $T\geq 0$, we get
\begin{align}\label{eq:lm:iff:1}
e^{v(x)}&=\mathbb{E}_x\left[e^{\int_0^{\tau_v \wedge T} g(X_s) ds+1_{\{\tau_v< T \}}G(X_{\tau_v})+1_{\{\tau_v\geq  T \}}v(X_T)}\right]\nonumber \\
&=\mathbb{E}_x\left[1_{\{\tau_v< T \}}e^{\int_0^{\tau_v \wedge T} g(X_s) ds+G(X_{\tau_v})}+1_{\{\tau_v\geq  T \}}e^{\int_0^{\tau_v \wedge T} g(X_s) ds+v(X_T)}\right].
\end{align}
Thus, recalling that $g(\cdot)\geq c>0$ and using Fatou Lemma, we get
\begin{equation}\label{eq:lm:iff:2}
\mathbb{E}_x\left[e^{\tau_v c}\right]\leq \mathbb{E}_x\left[\liminf_{T\to\infty}e^{(\tau_v\wedge T) c}\right]\leq \mathbb{E}_x\left[\liminf_{T\to\infty}e^{\int_0^{\tau_v\wedge T} g(X_s)ds}\right] \leq  e^{v(x)}<\infty,
\end{equation}
and, in particular, we get $\mathbb{P}_x[\tau_v<\infty]=1$. Thus, letting $T\to\infty$ in~\eqref{eq:lm:iff:1}, we get
\begin{align*}
e^{v(x)}=\lim_{T\to\infty}\mathbb{E}_x\left[1_{\{\tau_v< T \}}e^{\int_0^{\tau_v } g(X_s) ds+G(X_{\tau_v})}+1_{\{\tau_v\geq  T \}}e^{\int_0^{T} g(X_s) ds+v(X_T)}\right]\\
=\mathbb{E}_x\left[e^{\int_0^{\tau_v } g(X_s) ds+G(X_{\tau_v})}\right]+\lim_{T\to\infty}\mathbb{E}_x\left[1_{\{\tau_v\geq  T \}}e^{\int_0^{T} g(X_s) ds+v(X_T)}\right],
\end{align*}
where the second equality follows from monotone convergence theorem. This concludes the proof.
\end{proof}

Using Proposition~\ref{pr:Bellman_cont_opt} we get the closed-form formula for an optimal stopping time for the function $u$ under the continuity assumption. 

\begin{proposition}\label{pr:opt_u}
Let the function $u$ be given by~\eqref{eq:opt_stop_unb_cont}. Assume that $u$ is continuous. Then the stopping time 
\begin{equation}\label{eq:lm:opt_u}
\tau_u:=\inf\{t\geq 0: u(X_t)\geq G(X_t)\}
\end{equation}
 is optimal for $u$.
\end{proposition}
\begin{proof}
By Theorem~\ref{th:Bellman_cont} we know that $u$ satisfies the Bellman equation~\eqref{eq:Bellman_cont}. Also, as in~\eqref{eq:lm:iff:2}, we may show that $\mathbb{P}_x[\tau_u<\infty]=1$ for any $x\in E$. Thus, using  Proposition~\ref{pr:Bellman_cont_opt}, continuity of $u$, Fatou Lemma, and martingale property of the process $(z_u(t\wedge \tau_u))$, we get
\begin{align*}
e^{u(x)}&\leq \mathbb{E}_x\left[e^{\int_0^{\tau_u} g(X_s)ds+G(X_{\tau_u})}\right]=\mathbb{E}_x\left[e^{\int_0^{\tau_u} g(X_s)ds+u(X_{\tau_u})}\right]\\
& \leq \liminf_{t\to\infty}\mathbb{E}_x\left[e^{\int_0^{\tau_u\wedge t} g(X_s)ds+u(X_{\tau_u\wedge t})}\right]=\mathbb{E}_x\left[z_u(0)\right]=e^{u(x)}, \quad x\in E,
\end{align*}
which concludes the proof.
\end{proof}

\begin{remark}\label{rm:opt_u}
Recall that by Remark~\ref{rm:opt_tau_cont} we get that $\hat{\tau}$ from~\eqref{eq:hat_tau} is also optimal for $u$. However, the stopping time $\tau_u$ from~\eqref{eq:lm:opt_u} is smaller than $\hat{\tau}$. Indeed, noting that $u\equiv \underline{w}\geq \underline{w}_T$ for any $T\geq 0$, we get $\tau_u\leq \hat{\tau}$.
\end{remark}

Now, let us define the process
\begin{equation}\label{eq:z_unif}
Z (t):=\exp\left(\int_0^{t}g(X_s)ds+G(X_{t})\right), \quad t\geq 0.
\end{equation}
By analogy to Theorem~\ref{th:suff_disc} we may formulate a sufficient condition for $u\equiv w$. In particular, this gives uniqueness of a solution to~\eqref{eq:Bellman_cont}.

\begin{theorem}\label{th:suff_cont}
Let $u$ and $w$ be given by~\eqref{eq:opt_stop_unb_cont} and~\eqref{eq:opt_stop_unb_cont2}, respectively. Also, let $\hat{\tau}$ be given by~\eqref{eq:hat_tau}. Assume that the process $(Z(t\wedge \hat{\tau}))$, $t\geq 0$, given by~\eqref{eq:z_unif}, is uniformly integrable. Then,
\begin{enumerate}
\item\label{p:th:suff_cont:1} We get $u\equiv w$ and this function is continuous.
\item\label{p:th:suff_cont:2} The stopping time
\begin{equation}\label{eq:pr:suff_cont:over_tau}
\tau_u:=\inf\{t\geq 0: u(X_t)\geq G(X_t)\}
\end{equation}
is optimal for $u$. Also, we get $\tau_u=\lim_{T\to\infty} \overline{\tau}_T$, where $\overline{\tau}_T$ is given by~\eqref{eq:prop:coninuity_w_T:optstop}.
\item\label{p:th:suff_cont:3} The stopping time $\tau_u$ given by~\eqref{eq:pr:suff_cont:over_tau} is also optimal for $w$, i.e. we get
\begin{equation}\label{eq:opt_w_cont}
w(x)=\liminf_{T\to\infty}\ln\mathbb{E}_x\left[e^{\int_0^{\tau_u\wedge T} g(X_s)ds+G(X_{\tau_u\wedge T})}\right], \quad x\in E.
\end{equation}
\end{enumerate}
\end{theorem}
\begin{proof}
For transparency, we prove the claims point by point.

\medskip
\noindent
\textit{Proof of~\eqref{p:th:suff_cont:1}.} Recalling that by Remark~\ref{rm:opt_tau_cont} the stopping time $\hat{\tau}$ given by~\eqref{eq:hat_tau} is optimal for $u$ and using uniform integrability of $(Z(t\wedge \hat{\tau}))$, $t\geq 0$, for any $x\in E$, we get
\begin{align*}
e^{w(x)}\leq \lim_{T\to\infty} \E_x\left[e^{\int_0^{\hat\tau\wedge T} g(X_s)ds+G(X_{\hat\tau\wedge T})}\right]&=\E_x\left[e^{\int_0^{\hat\tau} g(X_s)ds+G(X_{\hat\tau})}\right]=e^{u(x)}.
\end{align*}
Recalling that we always get $u\leq w$, we conclude the proof of $u\equiv w$. Continuity follows from lower semicontinuity of $u$ and upper semicontinuity of $w$; see Theorem~\ref{th:equiv_cont} for details.

\medskip
\noindent
\textit{Proof of~\eqref{p:th:suff_cont:2}.} Note that optimality of $\tau_u$ follows from Proposition~\ref{pr:opt_u} and the fact that $u\equiv w$. Let us now show that
\begin{equation}\label{eq:th:suff_cont:1}
\tau_u=\lim_{T\to\infty} \overline{\tau}_T,
\end{equation}
Recalling Proposition~\ref{pr:finite_hor}, we get that the map $T\mapsto \overline{\tau}_T$ is increasing, hence the limit $\overline{\tau}:=\lim_{T\to\infty} \overline{\tau}_T$ is well-defined. Also, recalling that from~\eqref{eq:lm:iff:2} we get $\mathbb{P}_x\left[\tau_u<\infty\right]=1$ and using the fact that $u\equiv w\equiv \overline{w}\leq \overline{w}_T$ for any $T\geq 0$, on the event $\{\tau_u\leq T\}$ we get
\[
\overline{w}_{T-\tau_u}(X_{\tau_u})\geq u(X_{\tau_u})\geq G(X_{\tau_u}).
\]
Thus, we get $\overline{\tau}_T\leq \tau_u\wedge T$, hence, letting $T\to\infty$, we get $\overline{\tau}\leq \tau_u$. In particular, we get $\mathbb{P}_x\left[\overline{\tau}<\infty\right]=1$, $x\in E$. Also, recalling joint continuity of $(T,x)\mapsto \overline{w}_T(x)$, we get 
\begin{equation}\label{eq:th:suff_cont:2}
\overline{w}_{T-\overline{\tau}_T}(X_{\overline{\tau}_T})=G(X_{\overline{\tau}_T}).
\end{equation} We show that this implies $u(X_{\overline{\tau}})=G(X_{\overline{\tau}})$ and consequently $\tau_u\leq \overline{\tau}$. First, note that from a.s. finiteness of $\overline{\tau}$, we get $(T-\overline{\tau}_T)\to \infty$ as $T\to\infty$. Second, note that for any $T_n\to\infty$ and $x_n\to x$, we get
\[
|\overline{w}_{T_n}(x_n)-\overline{w}(x)|\leq |\overline{w}_{T_n}(x_n)-\overline{w}(x_n)|+|\overline{w}(x_n)-\overline{w}(x)|\to 0, \quad n\to\infty;
\]
this follows from Dini's theorem combined with the fact that $(\overline{w}_{T_n})$ is a sequence of continuous functions converging monotonically to the continuous function $\overline{w}$. Thus, letting $T\to\infty$ in~\eqref{eq:th:suff_cont:2}, we get $\overline{w}(X_{\overline{\tau}})=G(X_{\overline{\tau}})$, which combined with the fact that $u\equiv w \equiv \overline{w}$ concludes the proof of this part.

\medskip
\noindent
\textit{Proof of~\eqref{p:th:suff_cont:3}.} To show~\eqref{eq:opt_w_cont} it is enough to prove uniform integrability of $(Z(t\wedge\tau_u))$, $t\geq 0$, and use~\eqref{p:th:suff_cont:2}. Recalling that $\underline{w}\geq \underline{w}_T$ for any $T\geq 0$, on the set $\{\underline{\tau}_T<T\}$, we get $\underline{w}(X_{\underline{\tau}_T})\geq \underline{w}_{T-\underline{\tau}_T}(X_{\underline{\tau}_T})\geq G(X_{\underline{\tau}_T})$. Thus, letting $T\to\infty$, using continuity of $\underline{w}\equiv u$, and recalling that $\underline{w}\leq G$, we get $\underline{w}(X_{\hat{\tau}})=G(X_{\hat{\tau}})$ and consequently 
\begin{equation}\label{eq:pr:suff_cont:1}
\tau_u\leq \hat\tau.
\end{equation}
From Lemma~\ref{lm:convention} and uniform integrability of $(Z(t\wedge\hat{\tau}))$, $t\geq 0$, for any $x\in E$, we get $\liminf_{T\to\infty} \mathbb{E}_x\left[1_{\{\hat{\tau}>T \}} Z_T \right]=0$. Hence, using~\eqref{eq:pr:suff_cont:1}, for any $x\in E$, we also get $\liminf_{T\to\infty} \mathbb{E}_x\left[ 1_{\{ \overline\tau>T \}} Z_T \right]=0$ and, again by Lemma~\ref{lm:convention}, we conclude the proof of uniform integrability of $Z(t\wedge\tau_u)$, $t\geq 0$. Thus, recalling~\eqref{p:th:suff_cont:1} and~\eqref{p:th:suff_cont:2}, for any $x\in E$, we get
\[
e^{w(x)}=e^{u(x)}=\mathbb{E}_x\left[e^{\int_0^{\tau_u} g(X_s)ds+G(X_{\tau_u})}\right]=\lim_{T\to\infty} \mathbb{E}_x\left[e^{\int_0^{\tau_u\wedge T} g(X_s)ds+G(X_{\tau_u\wedge T})}\right],
\]
which concludes the proof.
\end{proof}

\begin{remark} In Theorem~\ref{th:suff_cont} continuity of $u$ was a consequence of the identity $u\equiv w$. However, if we know in advance that $u$ is continuous, we may obtain the results of Theorem~\ref{th:suff_cont} under weaker conditions. Namely, following the proof of Theorem~\ref{th:suff_cont}, we can see that, assuming continuity of $u$, one may replace uniform integrability of $Z(t\wedge\hat{\tau})$, $t\geq 0$, by uniform integrability of $Z(t\wedge{\tau_u})$, $t\geq 0$, where $\tau_u:=\inf\{t\geq 0:u(X_t)\geq G(X_t)\}$. Note that by Remark~\ref{rm:opt_u} the latter condition is less restrictive as $\tau_u\leq \hat{\tau}$.
\end{remark}

If the function $G$ is bounded, using Theorem~\ref{th:suff_cont} we may recover the results from~\cite{JelPitSte2019a}; see Theorem 15 therein.
\begin{corollary}\label{cor:identity_bounded}
If $G$ is bounded, then $u\equiv w$ and this function is continuous.
\end{corollary}
\begin{proof}
Recalling~\eqref{eq:hat_tau} and the following discussion, for any $x\in E$,  we get that for $\mathbb{P}_x$ almost all $\omega \in \Omega$, starting from some $n\in \mathbb{N}$ (depending on $\omega$), the sequence $(\underline{\tau}_n(\omega))$ is non-increasing. Thus, using right-continuity of $X$, we get $G(X_{\hat{\tau}}) = \lim_{n\to\infty} 1_{\{\underline{\tau}_n<n \}}G(X_{\underline{\tau}_n})$. Consequently, recalling non-negativity of $G$, Proposition~\ref{pr:finite_hor}, and using Fatou Lemma, for any $x\in E$, we get
\begin{align*}
\mathbb{E}_x\left[e^{\int_0^{\hat\tau} g(X_s) ds}\right]&\leq \mathbb{E}_x\left[e^{\int_0^{\hat\tau} g(X_s) ds+G(X_{\hat{\tau}})}\right] \\
&= \mathbb{E}_x\left[\lim_{n\to\infty}e^{\int_0^{\underline\tau_n} g(X_s) ds+1_{\{\underline{\tau}_n<n \}}G(X_{\underline{\tau}_n})}\right]\\
&\leq \lim_{n\to\infty}\mathbb{E}_x\left[e^{\int_0^{\underline\tau_n} g(X_s) ds+1_{\{\underline{\tau}_n<n \}}G(X_{\underline{\tau}_n})}\right]\\
&=\lim_{n\to\infty}e^{\underline{w}_n(x)}\leq e^{G(x)}<\infty.
\end{align*}
Combining this with the inequality $Z(t\wedge \hat{\tau})\leq e^{\int_0^{\hat\tau} g(X_s) ds} e^{\Vert G\Vert}$, $t\geq 0$, we get that by bounded convergence theorem the process $(Z(t\wedge \hat{\tau}))$, $t\geq 0$, is uniformly integrable. Consequently, using Theorem~\ref{th:suff_cont} we conclude the proof.
\end{proof}

\section{Reference examples}\label{S:app2}

In this section we provide a series of examples illustrating our assumptions and results. In particular, we provide a more general criterion for Assumptions~\eqref{IF}--\eqref{Feller}. Also, we show explicit formulae for multiple solutions to the Bellman equation.

\subsection{Examples for Assumptions~\eqref{IF}--\eqref{Feller}}\label{SS:assumptions}

In this section we comment on Assumptions~\eqref{IF} and~\eqref{Feller}. We show that they may be deduced from a more general condition:
\begin{enumerate} 
\item[(\namedlabel{IF_2}{$\mathcal{B}1$})] For any $T\geq 0$ and a compact set $K\subseteq E$ we get 
\[
\lim_{m\to\infty}\sup_{x\in K} \mathbb{E}_x\left[\zeta_T 1_{\{\zeta_T\geq m\}}\right]=0,
\]
where $\zeta_T=\sup_{t\in [0,T]} e^{G(X_t)}$.
\end{enumerate}
Condition~\eqref{IF_2} may be seen as a stronger form of integrability for $\zeta_T$. Namely, it requires that the tail of $\zeta_T$ is $\mathbb{P}_x$-integrable uniformly in $x$ from compact set.  Exemplary dynamics satisfying~\eqref{IF_2} is shown in Example~\ref{ex4}.

Let us now show that~\eqref{IF_2} implies~\eqref{IF} and~\eqref{Feller}.
\begin{lemma}\label{lm:assumpt_rel}
Assume~\eqref{IF_2}. Then~\eqref{IF} and~\eqref{Feller} hold.
\end{lemma}
\begin{proof}

For~\eqref{IF}, it is enough to note that for any $T\geq 0$, $x\in E$, and sufficiently large $m\in \mathbb{N}$, we get
\[
\mathbb{E}_x\left[\zeta_T\right]= \mathbb{E}_x\left[ \zeta_T 1_{\{\zeta_T< m\}}\right]+\mathbb{E}_x\left[ \zeta_T 1_{\{\zeta_T\geq  m\}}\right]\leq m+1<\infty.
\] 

For~\eqref{Feller}, let $T\geq 0$, $x\in E$, $(x_n)\to x$, and $h:E\mapsto \mathbb{R}_+$ be continuous and such that $h(\cdot)\leq G(\cdot)$. Let $\Gamma\subseteq E$ be a compact set satisfying $x\in \Gamma$ and $(x_n)\subset \Gamma$. We get
\begin{multline*}
\left|\mathbb{E}_x\left[e^{\int_0^T g(X_s)+h(X_T)}\right]-\mathbb{E}_{x_n}\left[e^{\int_0^T g(X_s)+h(X_T)}\right]\right| \\
\leq \left|\mathbb{E}_x\left[e^{\int_0^T g(X_s)+h(X_T)\wedge m}\right]-\mathbb{E}_{x_n}\left[e^{\int_0^T g(X_s)+h(X_T)\wedge m}\right]\right| \nonumber\\
 + 2\sup_{y\in \Gamma} \left|\mathbb{E}_y\left[e^{\int_0^T g(X_s)+h(X_T)}\right]-\mathbb{E}_{y}\left[e^{\int_0^T g(X_s)+h(X_T)\wedge m}\right]\right|.
\end{multline*}
Also, combining Lemma 4 from~\cite[Section II.5]{GikSko2004} and Corollary 2.2 from~\cite{PalSte2010}, we get that the map $x\mapsto \mathbb{E}_x\left[e^{\int_0^T g(X_s)+h(X_T)\wedge m}\right]$ is continuous for any $m\in \mathbb{N}$. Thus, to conclude the proof it is enough to show that 
\begin{equation}\label{eq:lm:assumpt_rel:1}
\sup_{y\in \Gamma} \left|\mathbb{E}_y\left[e^{\int_0^T g(X_s)+h(X_T)}\right]-\mathbb{E}_{y}\left[e^{\int_0^T g(X_s)+h(X_T)\wedge m}\right]\right|\to 0, \quad m\to \infty.
\end{equation}
Using~\eqref{IF_2}, for any $\varepsilon >0$ and sufficiently big $m\in \mathbb{N}$, we get
\begin{align*}
\sup_{y\in \Gamma} \mathbb{E}_y\left[e^{\int_0^T g(X_s)}\left|e^{h(X_T)}-e^{h(X_T)\wedge m}\right|\right]& \leq 2\sup_{y\in \Gamma} \mathbb{E}_y\left[e^{\int_0^T g(X_s)}e^{h(X_T)}1_{\{h(X_T)\geq m\}}\right]  \nonumber\\
&\leq 2e^{T\Vert g\Vert} \sup_{y\in \Gamma} \mathbb{E}_y \left[\zeta_T 1_{\{\zeta_T\geq m\}}\right]\leq \varepsilon.
\end{align*}
Thus, we get~\eqref{eq:lm:assumpt_rel:1}, which concludes the proof.
\end{proof}

Let us now show the exemplary dynamics satisfying Condition~\eqref{IF_2}.

\begin{example}\label{ex4}
Let $E=\mathbb{R}$, $G(x)=|x|$, and the process $(X_t)$ be a Brownian motion. Also, let $K\subseteq E$ be a compact set and $L_K:=\sup_{x\in K} |x|$. Note that under $\mathbb{P}_x$ we get $X_t=x+W_t$, where $W$ is a standard Brownian motion (starting from $0$). For the notational convenience, for any $T\geq 0$, we set $\zeta_T:=\sup_{t\in [0,T]} e^{|X_t|}$ and $S_T:=\sup_{t\in [0,T]} |W_t|$, $T\geq 0$. We show that for any $T\geq 0$ we get
\begin{equation}\label{eq:ex4:0.5}
\lim_{n\to\infty} \sup_{x\in K} \mathbb{E}_x\left[\zeta_T 1_{\{\zeta_T\geq e^n\}}\right] =0.
\end{equation}

Note that, for any $x\in E$, $T\geq 0$, and $n\in \mathbb{N}$, we get
\begin{align}\label{eq:ex4:0.75}
\sup_{x\in K} \mathbb{E}_x\left[\zeta_T 1_{\{\zeta_T\geq e^n\}}\right]&=\sup_{x\in K} \mathbb{E}_x\left[\sup_{t\in [0,T]}e^{|x+W_t|} 1_{\{\sup_{t\in [0,T]}|x+W_t|\geq n\}}\right]\nonumber\\
& \leq \sup_{x\in K} e^{|x|} \mathbb{E}_x\left[e^{S_T} 1_{\{ S_T\geq n-L_K\}}\right].
\end{align}
Moreover, we get that $\mathbb{E}_x\left[e^{S_T} 1_{\{ S_T\geq n-L_K\}}\right]$ is independent of $x\in E$. Thus, noting that $\sup_{x\in K} e^{|x|}<\infty$, to conclude the proof of~\eqref{eq:ex4:0.5}, it is enough to show 
\begin{equation}\label{eq:ex4:1}
\mathbb{E}_0\left[e^{S_T} \right]<\infty.
\end{equation}
Indeed, noting that $\mathbb{E}_0\left[e^{S_T} 1_{\{ S_T< n-L_K\}}\right]$ converges increasingly to $\mathbb{E}_0\left[e^{S_T} \right]$ as $n\to\infty$, and 
\[
\mathbb{E}_0\left[e^{S_T} \right]=\mathbb{E}_0\left[e^{S_T} 1_{\{ S_T< n-L_K\}}\right]+\mathbb{E}_0\left[e^{S_T} 1_{\{ S_T\geq n-L_K\}}\right],
\]
from~\eqref{eq:ex4:1} we get $\lim_{n\to\infty} \mathbb{E}_0\left[e^{S_T} 1_{\{ S_T\geq n-L_K\}}\right]=0$, which together with~\eqref{eq:ex4:0.75} implies~\eqref{eq:ex4:0.5}.

Let us now show~\eqref{eq:ex4:1}. Recalling that $(-W)$ is also a Brownian motion, we get
\[
\mathbb{E}_0\left[e^{S_T} \right]\leq \mathbb{E}_0\left[e^{\max\left( \sup_{t\in [0,T]} W_t, \sup_{t\in [0,T]} (-W_t)\right)}\right]\leq 2 \mathbb{E}_0\left[e^{\sup_{t\in [0,T]} W_t}\right].
\]
Recall that by reflection principle the distribution of $\sup_{t\in [0,T]} W_t$ is equal to the distribution of $|W_T|$; see e.g. Proposition 3.7 in~\cite[Chapter III]{RevYor1999} for details. Thus, we get
\[
\mathbb{E}_0\left[e^{S_T} \right]\leq 2 \mathbb{E}_0\left[e^{|W_T|}\right]<\infty,
\]
which concludes the proof.
\end{example}

\subsection{Examples for the Bellman equation}\label{SS:Bellman}

In this section we provide a series of computable examples related to the Bellman equation. In particular, we show a dynamics with a non-unique solution to this equation.

First, we show an example, where there is a strict inequality between the maps $u$ and $w$ given by~\eqref{eq:opt_stop_unb_disc} and~\eqref{eq:opt_stop_unb_disc2}. Recall that we already showed $u\leq w$.

\begin{example}\label{ex:1}
Let $E=\{1,2,3,\ldots\}$, $g\equiv c>0$ and $G(x)=x$, $x\in E$. Let $(X_n)_{n\in \mathbb{N}}$ be an i.i.d. sequence of discrete Pareto random variables, i.e.
\[
\mathbb{P}[X_n=k]=\frac{1}{C k^2}, \quad n\in \mathbb{N},\, k\in E,
\]
where $C:=\sum_{k=1}^\infty \frac{1}{k^2}=\frac{\pi^2}{6}$ is a normalizing constant. Recalling \eqref{eq:opt_stop_unb_disc2} and~\eqref{eq:opt_stop_unb_disc}, let us consider
\begin{align*}
u(x) & :=\inf_{\tau\in \mathcal{T}_0} \ln \mathbb{E}_x\left[e^{c\tau+X_{\tau}}\right], \quad x\in E;\\
w(x) & :=\inf_{\tau\in \mathcal{T}_0} \liminf_{n\to\infty} \ln \mathbb{E}_x\left[e^{c(\tau\wedge n)+X_{\tau\wedge n}}\right], \quad x\in E.
\end{align*}
Recalling~\eqref{eq:ineq_u_w}, we get $u(x)\leq w(x)$, $x\in E$. Let us show that this inequality may be strict.

First, we show that $w(x)=x$, $x\in E$. Recalling Theorem~\ref{th:und_over_prob}, we get 
\[
w(x)=\lim_{n\to\infty}\overline{w}_n(x),
\]
where $\overline{w}_n(x):=\inf_{\tau\leq n} \ln \mathbb{E}_x\left[e^{c\tau+X_{\tau}}\right]$, $x\in E$. Also, using Proposition~\ref{pr:disc_und_over}, for any $n\in \mathbb{N}$ and $x\in E$, we get $e^{\overline{w}_{n+1}(x)}=Se^{\overline{w}_{n}}(x)$, where the operator $S$ is given by $Sh(x):=e^x\wedge e^c \mathbb{E}_x[h(X_1)]$ and $\overline{w}_0(x)=x$. Noting that $\mathbb{E}_x[e^{X_1}]=+\infty$, $x\in E$, inductively we get $\overline{w}_n(x)=x$, $x\in E$, and consequently $w(x)=x$, $x\in E$.

Second, note that for $x\in E\setminus\{1\}$, $\tau_1:=\inf\{n\geq 0:X_n=1\}$, $p_1:=\mathbb{P}[X_1=1]$, and $c>0$ satisfying $c<-\ln(1-p_1)\approx 0.94$, we get
\begin{align*}
e^{u(x)} \leq \mathbb{E}_x\left[e^{c\tau_1+X_{\tau_1}}\right]=e\mathbb{E}_x\left[e^{c\tau_1}\right]&=e\sum_{k=1}^{\infty} e^{kc} p_1(1-p_1)^{k-1}\\
&= p_1 e^{c+1}\frac{1}{1-e^c(1-p_1)}=:B<\infty.
\end{align*}
Consequently, for $x>\ln B$, we get
\[
u(x)\leq \ln B<x=w(x),
\]
thus, there is a strict inequality between $u$ and $w$.

To better explain this situation, we directly show that the process 
\[
Z_{T\wedge\tau_1}:=e^{c\tau_1\wedge T+X_{\tau_1\wedge T}}, \quad T\in \mathbb{N}
\] 
is not uniformly integrable, cf. Lemma~\ref{lm:convention}. It is enough to show that
\[
L:=\lim_{n\to\infty} \sup_{T\in \mathbb{N}} \mathbb{E}\left[e^{c\tau_1\wedge T+X_{\tau_1\wedge T}}1_{\{e^{c\tau_1\wedge T+X_{\tau_1\wedge T}}\geq e^n \}}\right]=+\infty.
\]
Note that
\begin{align*}
L & \geq \lim_{n\to\infty} \sup_{T\in \mathbb{N}} e^n \mathbb{P}\left[c\tau_1\wedge T+X_{\tau_1\wedge T} \geq  n \right]\\
& = \lim_{n\to\infty} \sup_{T\in \mathbb{N}} e^n \left( \mathbb{P}\left[\tau_1 \leq T, c\tau_1+X_{\tau_1} \geq  n \right]+ \mathbb{P}\left[\tau_1 > T, cT+X_{T} \geq  n \right]\right)\\
& \geq \lim_{n\to\infty} \sup_{T\in \mathbb{N}} e^n \mathbb{P}\left[\tau_1 > T, X_{T} \geq  n-cT \right].
\end{align*}
Thus, setting $A:=\{1\}\subseteq E$ and for any $n\in \mathbb{N}$ setting $T=[0.5 n]$, where $[x]$ stands for the integer part of $x\in\mathbb{R}$, we get
\begin{align*}
L &\geq \lim_{n\to\infty} e^n \mathbb{P}\left[\tau_1 > [0.5 n], X_{[0.5 n]} \geq  n-c[0.5 n] \right]\\
& \geq \lim_{n\to\infty} e^n \mathbb{P}\left[X_1\in A^c, \ldots, X_{[0.5 n]-1}\in A^c, X_{[0.5 n]} = [n-c[0.5 n]]+1\right]\\
& = \lim_{n\to\infty} e^n (1-p_1)^{[0.5 n]-1} \frac{1}{C ([n-c[0.5 n]]+1)^2}.
\end{align*}
Let $a_n:=e^n (1-p_1)^{[0.5 n]-1} \frac{1}{C ([n-c[0.5 n]]+1)^2}$, $n\in \mathbb{N}$, and note that for $b_n:=e^n (1-p_1)^{0.5 n-1} \frac{1}{C n^2(1-0.5 c)^2}$, $n\in \mathbb{N}$, we get $\lim_{n\to\infty} \frac{a_n}{b_n}=1$. Also, we get 
\[
\lim_{n\to\infty} \frac{b_{n+1}}{b_n}=\lim_{n\to\infty} e(1-p_1)^{0.5} \frac{n^2}{(n+1)^2}=e(1-p_1)^{0.5}.
\]
Thus, noting that $e(1-p_1)^{0.5}\approx 1.7>1$, we get $b_n\to\infty$, hence $a_n\to \infty$ and $L=+\infty$. Consequently, the process $(Z_{T\wedge\tau_1})$, ${T\in \mathbb{N}}$, is not uniformly integrable.
\end{example}

In the next example we show explicit formulae for distinct solutions to the Bellman equation in the discrete time setting.

\begin{example}\label{ex:3}
Let $E=[0,+\infty)\subset\mathbb{R}$, $g\equiv c>0$ and $G(x)=x$, $x\in E$. Let $\alpha\in [0,1]$ and $(X_n)_{n\in \mathbb{N}}$ be a time-homogeneous Markov process with a transition probability
\[
\mathbb{P}_x[X_1=0]=\alpha, \, \mathbb{P}_x[X_1=x+1]=1-\alpha, \quad x\in E.
\]
Recalling \eqref{eq:opt_stop_unb_disc} and~\eqref{eq:opt_stop_unb_disc2}, let us consider
\begin{align*}
u(x) & :=\inf_{\tau\in \mathcal{T}_0} \ln \mathbb{E}_x\left[e^{c\tau+X_{\tau}}\right], \quad x\in E;\\
w(x) & :=\inf_{\tau\in \mathcal{T}_0} \liminf_{n\to\infty} \ln\mathbb{E}_x\left[e^{c(\tau\wedge n)+X_{\tau\wedge n}}\right], \quad x\in E.
\end{align*}
Also, let $K:=\ln \left(\frac{\alpha e^c}{1-(1-\alpha)e^c}\right)$; note that this constant is well-defined if $(1-\alpha)e^c<1$. We show that within this model
\begin{itemize}
\item If $\alpha \in [0,1-e^{-c}]$, then $u(x)=x=w(x)$, $x\in E$;
\item If $\alpha \in (1-e^{-c},1-e^{-c-1}]$, then $u(x)=x\wedge K$ and $w(x)=x$, $x\in E$;
\item If $\alpha \in (1-e^{-c-1},1]$, then $u(x)=x\wedge K=w(x)$, $x\in E$.
\end{itemize}
In particular, recalling Theorem~\ref{th:und_over_prob}, for $\alpha \in (1-e^{-c},1-e^{-c-1}]$ we get two distinct solutions to the Bellman equation
\begin{equation}\label{eq:esc_Bellman}
e^{v(x)}=e^x\wedge e^c\left(\alpha e^{v(0)}+(1-\alpha)e^{v(x+1)}\right),\quad  x\in E.
\end{equation}
Namely, we get that both $u$ and $w$ satisfy~\eqref{eq:esc_Bellman}, but $u(x)<w(x)$ for $x>K$. In fact, in this case we may construct infinitely many solutions to~\eqref{eq:esc_Bellman}; see Remark~\ref{rm:ex3:2}. Also, it should be noted that for $\alpha \in (1-e^{-c-1},1]$ both functions $u$ and $w$ are bounded despite the fact that $G$ is unbounded from above.

Note that $u(x)=x$ corresponds to the situation when instantaneous stopping is optimal; similar relation holds for $w$. Thus, we can see that for $\alpha$ small enough (relative to $c$), immediate stopping is optimal. However, for sufficiently big $\alpha$ it is optimal to wait until the process returns to zero; see the argument below for details.

For transparency, we split the argument into four parts: (1) proof of $u(x)=x\wedge K$, $x\in E$, for $\alpha \in (1-e^{-c},1]$; (2) proof of $u(x)=x$, $x\in E$, for $\alpha \in [0,1-e^{-c}]$; (3) proof of $w(x)=x$, $x\in E$ for $\alpha \in [0,1-e^{-c-1}]$; (4) proof of $w(x)=x\wedge K$, $x\in E$ for $\alpha \in (1-e^{-c-1},1]$.

\medskip
\noindent
\textit{Part (1)} We show that $u(x)=x\wedge K$, $x\in E$, for $\alpha \in (1-e^{-c},1]$. Recalling Theorem~\ref{th:und_over_prob} it is enough to show that $\lim_{n\to\infty} \underline{w}_n(x)=x\wedge K$, $x\in E$, where the sequence $(\underline{w}_n)_{n\in \mathbb{N}}$ is recursively defined as
\[
\underline{w}_0(x):=0, \quad e^{\underline{w}_{n+1}(x)}:=e^x\wedge e^c(\alpha e^{\underline{w}_n(0)}+(1-\alpha) e^{\underline{w}_n(x+1)}),\quad n\in \mathbb{N},\, x\in E.
\]
Recalling Proposition~\ref{pr:disc_und_over}, for any $n\in \mathbb{N}$ and $x\in E$, we get $\underline{w}_n(x)\geq \underline{w}_0(x)=0$. Thus, noting that $e^c(\alpha e^{\underline{w}_n(0)}+(1-\alpha) e^{\underline{w}_n(1)})\geq e^c>1$, we get $\underline{w}_n(0)=0$ for any $n\in \mathbb{N}$, and consequently
\[
e^{\underline{w}_{n+1}(x)}=e^x\wedge e^c(\alpha +(1-\alpha) e^{\underline{w}_n(x+1)}),\quad n\in \mathbb{N},\, x\in E.
\]
Let us now show that
\begin{equation}\label{eq:w_2_int_rek}
e^{\underline{w}_{n+1}(x)}=e^x \wedge e^{c_{n+1}}, \quad n\in \mathbb{N},\, x\in E,
\end{equation}
where 
\begin{equation}\label{eq:w_2_c_n}
e^{c_n}:=\sum_{k=1}^{n-1} \alpha (1-\alpha)^{k-1} e^{kc}+(1-\alpha)^{n-1} e^{cn}, \quad n=1,2,\ldots
\end{equation}  
First, note that by direct calculation we get $e^{c_{n+1}}\geq e^{c_n}$ and recalling that $(1-\alpha)e^c<1$, we get $e^{c_n}\to e^K$ as $n\to\infty$. To show~\eqref{eq:w_2_int_rek}, we proceed by induction. For $n=1$, we get
\[
e^{\underline{w}_{1}(x)}= e^x\wedge e^c = e^x \wedge e^{c_1}, \quad x\in E.
\]
Let us now assume that the claim holds for some $n\geq 1$. Then, for $x+1\geq c_{n}$, by direct calculation, we get
\[
e^{\underline{w}_{n+1}(x)}= e^x\wedge e^c(\alpha +(1-\alpha) e^{c_n})=e^x \wedge e^{c_{n+1}}, \quad x\in E.
\] 
Also, for $x<c_{n}-1\leq c_{n+1}-1\leq K-1$, we get
\[
e^{\underline{w}_{n+1}(x)}=e^x\wedge e^c(\alpha +(1-\alpha) e^{x}).
\]
Thus, to conclude the proof it is enough to show $e^c(\alpha +(1-\alpha) e^{x})\geq e^x$ for $x\in [0,K-1]$. Let us define $h(x):=e^c (\alpha +(1-\alpha)e^{x+1})-e^x$, $x\in E$. Noting that  $h'(x)=e^x(e^{c+1} (1-\alpha)-1)$ we get that $h$ is monotonic. This together with the estimates
\begin{align*}
h(0)& =e^c \alpha+e^c(1-\alpha)e-1\geq e^c(\alpha+(1-\alpha))-1=e^c-1>0;\\
h(K-1)& =e^c\alpha+e^c(1-\alpha)e^{K} - e^{K-1}=\frac{\alpha e^c}{1-(1-\alpha) e^c	}(1-e^{-1})>0
\end{align*}
shows $h(x)\geq 0$ for $x\in [0,K-1]$. Thus, for $x<c_{n}-1\leq K-1$, we get
\[
e^{\underline{w}_{n+1}(x)}=e^x\wedge e^c(\alpha +(1-\alpha) e^{x})=e^x=e^x\wedge e^{c_{n+1}},
\]
which concludes the proof of~\eqref{eq:w_2_int_rek}. Letting $n\to\infty$ in~\eqref{eq:w_2_int_rek} and recalling Theorem~\ref{th:und_over_prob} we get $u(x)=x\wedge K$.

\medskip
\noindent
\textit{Part (2)} We show that $u(x)=x$, $x\in E$, for $\alpha \in [0,1-e^{-c}]$. Noting that $(1-\alpha)e^c\geq 1$ and recalling~\eqref{eq:w_2_c_n}, we get that $c_n\to\infty$ as $n\to \infty$. Thus, to conclude the proof it is enough to show $\underline{w}_{n+1}(x):=x\wedge c_{n+1}$, $n\in \mathbb{N}$, $x\in E$. As previously, for $x+1\geq c_{n}$, the claim follows from direct calculation. For $x+1< c_{n}$ let us define $h(x):=e^c (\alpha +(1-\alpha)e^{x+1})-e^x$, $x\in E$ and note that
\[
h'(x)=e^x(e^{c+1} (1-\alpha)-1)\geq 0, \quad x\in E,
\]
as $e^{c+1}(1-\alpha)>e^c(1-\alpha)\geq 1$. This, together with the inequality $h(0)>0$ shows $h(x)\geq 0$, $x\in E$. Consequently, we get$\underline{w}_{n+1}(x):=x\wedge c_{n+1}$, thus letting $n\to\infty$, we get $u(x)=x$, $x\in E$.

\medskip
\noindent
\textit{Part (3)} We show that $w(x)=x$, $x\in E$, for $\alpha \in [0,1-e^{-c-1}]$. Recalling Theorem~\ref{th:und_over_prob} it is enough to show $\lim_{n\to\infty} \overline{w}_n(x)=x$, $x\in E$, where the sequence $(\overline{w}_n)$ is recursively defined as
\[
\overline{w}_0(x):=x, \quad e^{\overline{w}_{n+1}(x)}:=e^x\wedge e^c(\alpha e^{\overline{w}_n(0)}+(1-\alpha) e^{\overline{w}_n(x+1)}),\quad n\in \mathbb{N}, x\in E.
\]
Noting that $\alpha \in [0,1-e^{-c-1}]$ implies $(e^{c+1} (1-\alpha)-1)\geq 0$, we get
\[
e^x(e^{c+1} (1-\alpha)-1)\geq -\alpha e^c, \quad x\in E.
\]
This inequality is equivalent to $e^c(\alpha+(1-\alpha)e^{x+1})\geq e^x$, $x\in E$, which implies
\begin{equation}\label{eq:ex3:eq1}
e^x\wedge e^c(\alpha e^0+(1-\alpha)e^{x+1}) = e^x, \quad x\in E.
\end{equation}
Using~\eqref{eq:ex3:eq1}, inductively we get $\overline{w}_n(x)=x$ for any $x\in E$. Thus $\lim_{n\to\infty}\overline{w}_n(x)=x$, $x\in E$, and recalling Theorem~\ref{th:und_over_prob} we get $w(x)=x$, $x\in E$.

\medskip
\noindent
\textit{Part (4)} We show that $w(x)=x\wedge K$, $x\in E$, for $\alpha \in (1-e^{-c-1},1]$. Recalling that in this case $u(x)=x\wedge K$ and $u(x)\leq w(x)$, $x\in E$, it is enough to show
\begin{equation}\label{eq:ex3:part4}
\liminf_{n\to\infty} \mathbb{E}_x\left[e^{c\tau_K\wedge n+X_{\tau_K\wedge n}}\right]=e^{x\wedge K}, \quad x\in E,
\end{equation}
where $\tau_K=\inf\{n\geq 0:X_n\in [0,K]\}$. For $x\in [0,K]$ we get $\mathbb{P}_x[\tau_K=0]=1$ and consequently $\mathbb{E}_x\left[e^{c\tau_{K}+X_{\tau_{K}}}\right]=e^x$. For $x>K$ we get 
\[
\mathbb{P}_x[\tau_K=\inf\{n\geq 0:X_n=0\}]=1.
\]
Thus, for $x>K$ and $n\geq 1$, we get
\begin{align*}
\mathbb{E}_x\left[e^{c\tau_K\wedge n+X_{\tau_K\wedge n}}\right] & = \sum_{k=1}^n \mathbb{E}_x\left[1_{\{\tau_K=k \}}e^{ck +X_{k}}\right]+\sum_{k=n+1}^\infty \mathbb{E}_x\left[1_{\{\tau_K=k \}}e^{cn+X_n}\right]\\
& = \sum_{k=1}^n \alpha (1-\alpha)^{k-1} e^{ck}+\sum_{k=n+1}^\infty \alpha(1-\alpha)^{k-1} e^{cn+x+n}.
\end{align*}
Noting that $\sum_{k=n+1}^\infty \alpha (1-\alpha)^{k-1} =(1-\alpha)^n$ and $(1-\alpha)^n e^{n(c+1)}\to 0$ as $n\to\infty$, we get
\[
\liminf_{n\to\infty}\mathbb{E}_x\left[e^{c\tau_K\wedge n+X_{\tau_K\wedge n}}\right] = \lim_{n\to\infty} \sum_{k=1}^n \alpha (1-\alpha)^{k-1} e^{ck} = e^K,\quad x>K,
\]
which concludes the proof of~\eqref{eq:ex3:part4}.
\end{example}

\begin{remark}\label{rm:ex3}
Let $\underline{\tau}:=\inf\{n\in \mathbb{N}: u(X_n)=G(X_n)\}$ and let 
\[
Z_n:=\exp\left(\sum_{i=0}^{n-1}g(X_i)+G(X_n)\right).
\]
Using the argument leading to~\eqref{eq:ex3:part4} we may show that the process $(Z_{n\wedge \underline{\tau}})$, ${n\in \mathbb{N}}$, is uniformly integrable if and only if $\alpha\in [0,1-e^{-c}]\cup (1-e^{-c-1},1]$. Thus, in this case the condition from Theorem~\ref{th:suff_disc} is also necessary for the equality $u\equiv w$.

Next, let $\overline{\tau}:=\inf\{n\in \mathbb{N}: w(X_n)=G(X_n)\}$. One may show that the process $(Z_{n\wedge \overline{\tau}})$, ${n\in \mathbb{N}}$, is uniformly integrable for any $\alpha\in [0,1]$. In particular, for $\alpha \in (1-e^{-c},1-e^{-c-1}]$, we get that uniform integrability of $(Z_{n\wedge \overline{\tau}})$, ${n\in \mathbb{N}}$, does not imply the equality of $u$ and $w$; see Remark~\ref{rm:lm:suff_disc:2}.
\end{remark}

\begin{remark}\label{rm:ex3:2}
Consider the model from Example~\ref{ex:3} with $\alpha \in (1-e^{-c},1-e^{-c-1}]$. Define the function $v:E\mapsto \mathbb{R}$ by
\[v(x):=
\begin{cases}
x, & x\in [0,K]\cup \mathbb{N},\\
K, & \text{otherwise}.
\end{cases}
\]
We show that $v$ is also a solution to the Bellman equation~\eqref{eq:esc_Bellman}. Indeed, noting that $v(x)=w(x)$ for $x\in \mathbb{N}$, where $w(x):=x$, $x\in E$, and recalling that $w$ is a solution to~\eqref{eq:esc_Bellman}, we get
\begin{align*}
e^{v(x)}=e^{w(x)}&=e^x\wedge e^c\left(\alpha e^{w(0)}+(1-\alpha)e^{w(x+1)}\right)\\
& =e^x\wedge e^c\left(\alpha e^{v(0)}+(1-\alpha)e^{v(x+1)}\right), \quad x\in \mathbb{N}.
\end{align*}
Similarly, noting that $v(x)=u(x)$ for $x\in E\setminus \mathbb{N}$, where $u(x):=x\wedge K$, $x\in E$, we get
\[
e^{v(x)}=e^x\wedge e^c\left(\alpha e^{v(0)}+(1-\alpha)e^{v(x+1)}\right), \quad x\in E\setminus \mathbb{N}.
\]
Consequently, $v$ is a solution to~\eqref{eq:esc_Bellman} and it is different from $u$ and $w$; cf. Theorem~\ref{th:und_over_prob}. Also, note that $v$ is discontinuous. In fact, using similar logic we may construct infinitely many (discontinuous) solutions to~\eqref{eq:esc_Bellman}. 
\end{remark}

We conclude this section with the example for the non-uniqueness of a solution to the continuous time Bellman equation.

\begin{example}\label{ex:5}
In this example we use the dynamics from Example~\ref{ex:3} to get a piecewise deterministic (piecewise constant) continuous time Markov process $X$ on the state space $E:=[0,+\infty)$. In a nutshell, under the measure $\mathbb{P}_x$, the process $X$ starts at $x\in E$ and stays at this state up to the exponentially distributed time $\tau_1$. At $\tau_1$, the process is subject to the immediate jump, with after jump state equals to $0$ with probability $\alpha$ and equals to $x+1$ with probability $(1-\alpha)$. Then, the process stays at the new state with independent exponentially distributed time and the procedure repeats.

Let us now provide more details on the process construction. First, let $(Y_n)$ be a discrete time Markov process with dynamics studied in Example~\ref{ex:3}, i.e. 
\[
\mathbb{P}_x[Y_1=0]=\alpha, \, \mathbb{P}_x[Y_1=x+1]=1-\alpha, \quad x\in E,
\]
for some $\alpha\in [0,1]$. Also, let $(\tau_n)_{n=1}^\infty$ be an increasing sequence of non-negative random variables. We assume that under any $\mathbb{P}_x$, $x\in E$, the increments $(\tau_{n+1}-\tau_n)$, $n\in \mathbb{N}$, are exponentially distributed with (common) parameter $\lambda>0$; note that here we follow the convention $\tau_0\equiv 0$. Also, we assume that under any $\mathbb{P}_x$, $x\in E$, jump times $(\tau_n)$ are independent of $(Y_n)$. Finally, we define the process $X$ as $X_t:=Y_n$ for $t\in [\tau_n,\tau_{n+1})$. We refer to~\cite{Dav1993} for a more detailed discussion on the piecewise deterministic Markov processes.

By analogy to Example~\ref{ex:3}, we set $g\equiv d$ with $d\in(0,\lambda)$ and $G(x)=x$, $x\in E$. Also, we consider the continuous time optimal stopping problems
\begin{align}
u(x)&:=\inf_{\tau}\ln \mathbb{E}_x[e^{d\tau+X_{\tau}}], \quad x\in E.\label{ex5:ustop}\\
w(x)&:=\inf_{\tau}\liminf_{T\to\infty}\ln \mathbb{E}_x[e^{d(\tau\wedge T)+X_{\tau\wedge T}}], \quad x\in E.\label{ex5:wstop}
\end{align}
Due to the non-negativity of $d$, it is optimal to stop the process only at the times when the process is subject to a jump. Thus, the problem may be embedded in the discrete-time setting with the corresponding Bellman equation of the form
\begin{equation}\label{ex5:eq:Bellman_cont}
e^{v(x)}=e^x\wedge \mathbb{E}_x\left[e^{d\tau_1+v(X_{\tau_1})}\right], \quad x\in E.
\end{equation}
Using independence of $(Y_n)$ and $(\tau_n)$ and the fact that $\tau_1$ is exponentially distributed, for any $x\in E$, we get
\[
\mathbb{E}_x\left[e^{d\tau_1+v(X_{\tau_1})}\right] =\mathbb{E}_x\left[e^{d\tau_1+v(Y_1)}\right]  = \int_0^\infty \lambda e^{-t(\lambda-d)} dt \left(\alpha e^{v(0)}+(1-\alpha)e^{v(x+1)}\right).
\]
Thus, Equation~\eqref{ex5:eq:Bellman_cont} could be rewritten as
\[
e^{v(x)}=e^x\wedge \frac{\lambda}{\lambda-d}\left(\alpha e^{v(0)} +(1-\alpha)e^{v(x+1)}\right), \quad x\in E.
\]
Note that setting $c:=\ln \lambda-\ln (\lambda-d) $, we get
\begin{equation}\label{ex5:eq:Bellman_disc}
e^{v(x)}=e^x\wedge e^c\left(\alpha e^{v(0)} +(1-\alpha)e^{v(x+1)}\right), \quad x\in E,
\end{equation}
which coincides with~\eqref{eq:esc_Bellman}. Thus, recalling the discussion in Example~\ref{ex:3}, we get the continuous time dynamics with multiple solutions to the corresponding Bellman equation. More specifically, using a suitable embedding, it can be shown that solutions to~\eqref{ex5:eq:Bellman_disc} satisfy
\begin{equation}\label{ex5:eq:Bellman_cont2}
e^{v(x)}=\inf_{\tau}\mathbb{E}_x\left[e^{(\tau \wedge t) d +1_{\{\tau< t \}}X_\tau+1_{\{\tau\geq  t \}}v(X_t)}\right], \quad t\geq 0,\,x\in E,
\end{equation}
which is a version of~\eqref{eq:Bellman_cont} corresponding to~\eqref{ex5:ustop} and~\eqref{ex5:wstop}. Since by Example~\ref{ex:3} we get multiple solutions to~\eqref{ex5:eq:Bellman_disc}, we also get multiple solutions to~\eqref{ex5:eq:Bellman_cont2}.
\end{example}


\appendix

\section{Deferred proofs}\label{S:proofs}
In this section we present the proof of Proposition~\ref{pr:finite_hor}. This is an extension of the results from~\cite{JelPitSte2019a}, where the function $G$ is assumed to be bounded from above; see Propositions 10 and 11 therein. Throughout this section we assume~\eqref{cost_functions}--\eqref{Feller}.

For any $n\in \mathbb{N}$ and $T\geq 0$ let us define bounded versions of~\eqref{eq:w^t_approx_from_below} and~\eqref{eq:w^T_approx_from_above} by
\begin{align}
\underline{v}_T^n(x) & :=\inf_{\tau\leq T} \ln\E_x\left[e^{\int_0^\tau g(X_s)ds + 1_{\{\tau<T\}}G(X_\tau)\wedge n }\right],\quad & T\geq 0, \, x\in E,\label{eq:opt_stop_finhor_bound2}\\
\overline{v}_T^n(x) & :=\inf_{\tau\leq T} \ln \E_x\left[e^{\int_0^{\tau} g(X_s)ds+G(X_{\tau})\wedge n}\right], \quad & T\geq 0, \, x\in E.\label{eq:opt_stop_finhor_bound}
\end{align}
We summarise the properties of $\underline{v}_T^n$ and $\overline{v}_T^n$ in the following lemma. For the proof, see Proposition 11 and Remark 12 from~\cite{JelPitSte2019a}.
\begin{lemma}\label{lm:opt_stop_finhor_bound}
Let $n\in \mathbb{N}$ and let the functions $\underline{v}_T^n$ and $\overline{v}_T^n$ be given by~\eqref{eq:opt_stop_finhor_bound2} and~\eqref{eq:opt_stop_finhor_bound}, respectively. Then
\begin{enumerate}
\item The function $(T,x)\mapsto \underline{v}_T^n(x)$ is jointly continuous. Moreover, 
\begin{equation}\label{eq:opt_stop_rule_finhor_bound2}
\underline{\tau}_T^n=\inf\{t\geq 0:\underline{v}_{T-t}^n(X_t)\geq G(X_t)\wedge n\}\wedge T
\end{equation}
is an optimal stopping time for $\underline{v}_T^n$.
\item The function $(T,x)\mapsto \overline{v}_T^n(x)$ is jointly continuous. Moreover, 
\begin{equation}\label{eq:opt_stop_rule_finhor_bound}
\overline{\tau}_T^n=\inf\{t\geq 0:\overline{v}_{T-t}^n(X_t)\geq G(X_t)\wedge n\} 
\end{equation}
is an optimal stopping time for $\overline{v}_T^n$.
\end{enumerate}
\end{lemma}

Let us now link the functions $\underline{v}_T^n$ and $\overline{v}_T^n$ with $\underline{w}_T$ and $\overline{w}_T$.
\begin{lemma}\label{lm:fin1}
Let the functions $\underline{w}_T$ and $\overline{w}_T$ be given by~\eqref{eq:w^t_approx_from_below} and~\eqref{eq:w^T_approx_from_above}, respectively. Also, let the sequences $(\underline{v}^n_T)$ and $(\overline{v}^n_T)$ be given by~\eqref{eq:opt_stop_finhor_bound2} and~\eqref{eq:opt_stop_finhor_bound}, respectively. Then, for any $x\in E$ and $T\geq 0$, we get
\begin{align*}
\underline{w}_T(x)=\lim_{n\to\infty} \underline{v}^n_T(x)\quad \text{and} \quad \overline{w}_T(x)=\lim_{n\to\infty} \overline{v}^n_T(x).
\end{align*}
\end{lemma}
\begin{proof}
We present the proof only for $\underline{w}_T$; the proof for $\overline{w}_T$ is analogous and omitted for brevity.

Let us fix $T\geq 0$ and $x\in E$. Also, let us define the family of events $A_n:=\{\sup_{t\in [0,T]} G(X_t)\leq n\}$, $n\in \mathbb{N}$. For any $n\in \mathbb{N}$ we get $A_n\subset A_{n+1}$. Moreover, using c\`adl\`ag property of $X$, continuity of $G$, and the fact that $T<\infty$, we get $\mathbb{P}_x\left[ \cup_{n=1}^\infty A_n\right]=1$.

Recalling Lemma~\ref{lm:opt_stop_finhor_bound} and using right continuity of $X$, on the event $\{\underline{\tau}_T^n<T\}$, we get $\underline{v}^n_{T-\underline{\tau}_T^n}(X_{\underline{\tau}_T^n})\geq G(X_{\underline{\tau}_T^n})\wedge n$. Thus, on the event $A_n\cap \{\underline{\tau}_T^n<T\}$ we get
\[
\underline{v}^{n+1}_{T-\underline{\tau}_T^n}(X_{\underline{\tau}_T^n})\geq \underline{v}^n_{T-\underline{\tau}_T^n}(X_{\underline{\tau}_T^n})\geq G(X_{\underline{\tau}_T^n})\wedge n=G(X_{\underline{\tau}_T^n})\geq G(X_{\underline{\tau}_T^n})\wedge (n+1),
\]
hence $\underline{\tau}_T^{n+1}\leq \underline{\tau}_T^n$ on $A_n\cap \{\underline{\tau}_T^n<T\}$. In fact, we get $\underline{\tau}_T^{n+1}\leq \underline{\tau}_T^n$ on $A_n$; this follows from the fact that on $A_n\cap \{\underline{\tau}_T^n=T\}$ directly from~\eqref{eq:opt_stop_rule_finhor_bound2} we get $\underline{\tau}_T^{n+1}\leq T=\underline{\tau}_T^n$. Thus, acting inductively, for any $k\geq 0$ we get $\underline{\tau}_T^{n+k+1}\leq \underline{\tau}_T^{n+k}$ on $A_n$. Thus, the limit $ \hat{\underline{\tau}}_T:=\lim_{n\to\infty}\underline{\tau}_T^{n}$ is well defined. Then, using right continuity of $X$, finiteness of $G$, and Fatou Lemma we get
\begin{align*}
e^{\underline{w}_T(x)}& \leq \mathbb{E}_x\left[e^{\int_0^{\hat{\underline{\tau}}_T} g(X_s)ds+1_{\{\hat{\underline{\tau}}_T<T \}}G(X_{\hat{\underline{\tau}}_T})}\right] \nonumber\\
& =\mathbb{E}_x\left[\lim_{n\to\infty}e^{\int_0^{\underline{\tau}_T^{n}} g(X_s)ds+1_{\{\underline{\tau}_T^{n}<T\}}G(X_{\underline{\tau}_T^{n}})\wedge n}\right] \nonumber\\
& \leq \liminf_{n\to\infty}\mathbb{E}_x\left[e^{\int_0^{\underline{\tau}_T^{n}} g(X_s)ds+1_{\{\underline{\tau}_T^{n}<T\}}G(X_{\underline{\tau}_T^{n}})\wedge n}\right]=\lim_{n\to\infty} e^{\underline{v}_T^n(x)}\leq e^{\underline{w}_T(x)},
\end{align*}
which concludes the proof.
\end{proof}

Let us now show a useful result characterising an optimal stopping time for the finite horizon stopping problem with possible discontinuity at the terminal point.
\begin{lemma}\label{lm:opt_fin}
Let $h:E\mapsto \mathbb{R}_+$ be a continuous function satisfying $h(\cdot)\leq G(\cdot)$. Also, for any $T\geq 0$, let us define
\[
v_T(x):=\inf_{\tau\leq T} \ln \mathbb{E}_x\left[e^{\int_0^{\tau} g(X_s)ds + 1_{\{\tau<T\}}G(X_{\tau})+1_{\{\tau=T \}}h(X_T) }\right], \quad x\in E.
\]
Assume that the map $(T,x)\mapsto v_T(x)$ is jointly continuous. Then, for any $T\geq 0$ the stopping time
\[
\tau_T:=\inf\{t\geq 0: v_{T-t}(X_t)\geq G(X_t)\}\wedge T
\]
is optimal for $v_T(x)$, $x\in E$. Moreover, for any $T\geq 0$ and $x\in E$, the process
\[
z_T(t):=e^{\int_0^{t\wedge T} g(X_s) ds+v_{T-t\wedge T}(X_{t\wedge T})}, \quad t\geq 0
\]
is a $\mathbb{P}_x$-submartingale and $(z_T(t\wedge {\tau}_T))$, $t\geq 0$, is a $\mathbb{P}_x$-martingale.
\end{lemma}
\begin{proof}
The argument is partially based on the third step of the proof of Proposition 11 in~\cite{JelPitSte2019a}. For transparency, we present it in detail.

We start with showing optimality of $\tau_T$. For $t\in [0,T]$, let us define
\begin{align*}
y_T(t)& :=e^{\int_0^{t} g(X_s)ds + 1_{\{t<T\}}G(X_{t})+1_{\{t=T \}}h(X_T) }.
\end{align*}
Using argument from \cite{Fak1971} one can show that $z_T$ is the Snell envelope of $y_T$. In particular, from Theorem 2 in~\cite{Fak1970} we get that $(z_T(t))$, $t\geq 0$, is a submartingale. Also, using Theorem 4 from~\cite{Fak1970}, we get that
\begin{equation*}
\tau^\vep_T:=\inf\left\{t\geq 0:  z_T(t)\geq -\vep + y_T(t)\right\}
\end{equation*}
is an $\vep$-optimal stopping time for $e^{v_T(x)}$, for any $\vep >0$, $T\geq 0$, and $x\in E$. Thus, setting
\begin{equation}\label{eq:def:tau.nth2}
{\hat{\tau}}^\vep_T:=\inf\left\{t\geq 0:  e^{v_{T-t}(X_t)}\geq (-\vep)\cdot e^{-\int_0^t g(X_s)ds} +e^{G(X_t)}\right\},
\end{equation}
we get $\tau^\vep_T={\hat{\tau}}^\vep_T\wedge T$. Now, noting that ${\hat{\tau}}^{\vep_1}_T\geq {\hat{\tau}}^{\vep_2}_T$, whenever $0\leq \vep_1\leq \vep_2$, we may define 
\[
{\hat{\tau}}_T:=\lim_{\vep \downarrow 0} {\hat{\tau}}^{\vep}_T\wedge T=\lim_{\vep \downarrow 0} \tau^\vep_T.
\]
Let us now show that ${\hat{\tau}}_T=\tau_T$. For any $\epsilon>0$, on the event $\{{\hat{\tau}}^\vep_T<T\}$, recalling \eqref{eq:def:tau.nth2}, continuity of $(T,x)\mapsto v_T(x)$ and $x\mapsto G(x)$, and right-continuity of $(X_t)$, we get
\begin{equation}\label{eq:calc22}
e^{v_{T-{\hat{\tau}}^\vep_T}\left(X_{{\hat{\tau}}^\vep_T}\right)}\geq (-\vep) \cdot e^{-\int_0^{{\hat{\tau}}^\vep_T}g(X_s)ds}+e^{G(X_{{\hat{\tau}}^\vep_T})}.
\end{equation}
Thus, on the set $\{{\hat\tau}_T<T\}$, letting $\vep\downarrow 0$ in~\eqref{eq:calc22}, we get $
e^{v_{T-\hat\tau_T}(X_{\hat\tau_T})}\geq e^{G(X_{\hat\tau_T})}$.
Since $v_T(x)\leq G(x)$, for any $x\in E$ and $T\geq 0$, on the event $\{{\hat\tau}_T<T\}$, we also get $v_{T-\hat\tau_T}(X_{{\hat\tau}_T})= G(X_{{\hat\tau}_T})$.
Recalling definition of ${\tau}_T$, we get ${\tau}_T\leq {\hat\tau}_T$. Noting that $\tau^\vep_T\leq  {\tau}_T$, for any $\vep>0$, and letting $\vep\to 0$, we get ${\tau}_T={\hat\tau}_T$.

Now we show that ${\tau}_T={\hat\tau}_T$ is optimal for $v_T$. Using Fatou Lemma  we get
\begin{align}
\lim_{\vep \to 0} (e^{v_T(x)}+\vep) & \geq \liminf_{\vep \to 0} \E_x \left[ e^{\int_0^{\tau^\vep_T} g(X_s)ds + 1_{\{\tau^\vep_T<T\}}G(X_{\tau^\vep_T}) +1_{\{\tau^\vep_T=T\}}h(X_T) } \right] \nonumber\\
& \geq \E_x\left[\liminf_{\vep \to 0} e^{\int_0^{\tau^\vep_T} g(X_s)ds + 1_{\{\tau^\vep_T<T\}}G(X_{\tau^\vep_T}) +1_{\{\tau^\vep_T=T\}}h(X_T) } \right].\label{eq:part1.222}
\end{align}
Note that $\int_0^{\tau^\vep_T} g(X_s)ds\to \int_0^{\hat\tau_T} g(X_s)ds$ as $\varepsilon \downarrow 0$. Also, recalling monotonicity of $\varepsilon \mapsto \tau^\vep_T$,  on the event $A:=\{\exists \varepsilon \colon \tau_T^\varepsilon =T\}$, we get 
\begin{multline*}
\lim_{\varepsilon \to 0} \left(1_{\{\tau^\vep_T<T\}}G(X_{\tau^\vep_T})+1_{\{\tau^\vep_T=T\}}h(X_T)\right)\\
=\lim_{\varepsilon \to 0} 1_{\{\tau^\vep_T=T\}}h(X_T)= 1_{\{\hat\tau_T=T\}}h(X_T)\\
=1_{\{\hat\tau_T<T\}}G(X_{\hat\tau_T})+1_{\{\hat\tau_T=T\}}h(X_T).
\end{multline*}
Similarly, using quasi-left continuity of $X$ and recalling that $G\geq h$, on the event $A^c=\{\forall \varepsilon \colon \tau_T^\varepsilon <T\}$, we get
\begin{multline*}
\lim_{\varepsilon \to 0} \left(1_{\{\tau^\vep_T<T\}}G(X_{\tau^\vep_T})+1_{\{\tau^\vep_T=T\}}h(X_T)\right)\\
=\lim_{\varepsilon \to 0} G(X_{\tau^\vep_T})=G(X_{\hat\tau_T})\\
\geq 1_{\{\hat\tau_T<T\}}G(X_{\hat\tau_T})+1_{\{\hat\tau=T\}}h(X_T).
\end{multline*}
Thus, from~\eqref{eq:part1.222}, we get
\[
\lim_{\vep \to 0} (e^{v_T(x)}+\vep) \geq \E_x\left[e^{\int_0^{\hat{\tau}_T} g(X_s)ds + 1_{\{\hat{\tau}_T<T\}}G(X_{\hat\tau_T}) +1_{\{\hat{\tau}_T=T\}}h(X_T) } \right]\geq e^{v_T(x)}
\]
and ${\tau}_T={\hat\tau}_T$ is optimal for $v_T$.

Finally, let us show martingale property of $(z_T(t\wedge {\tau}_T))$, $t\geq 0$.
Noting that for any $t\geq 0$ we get ${z}_T(t\wedge {\tau}_T)\leq e^{T \Vert g\Vert} \sup_{t\in [0,T]} e^{G(X_t)}$ and using~\eqref{IF}, we get that the process $(z_T(t\wedge {\tau}_T))$, $t\geq 0$, is uniformly integrable. In particular, recalling that ${\tau}_T$ is optimal for $v_T(x)$, $x\in E$, we get
\begin{align}\label{eq:th:w^t:mart1}
\mathbb{E}_x[z_T(0)]=e^{v_T(x)}&=\mathbb{E}_x\left[e^{\int_0^{{\tau}_T} g(X_s) ds+1_{\{{\tau}_T<T\}}G(X_{{\tau}_T})+1_{\{{\tau}_T=T\}}h(X_T)}\right]\nonumber\\
&=\mathbb{E}_x\left[e^{\int_0^{{\tau}_T} g(X_s) ds+v_{T-{\tau}_T}(X_{{\tau}_T})}\right]\nonumber\\
&=\mathbb{E}_x\left[\lim_{t\to\infty}e^{\int_0^{{\tau}_T\wedge t} g(X_s) ds+v_{T-{\tau}_T\wedge t}(X_{{\tau}_T\wedge t})}\right]\nonumber\\
&=\lim_{t\to\infty} \mathbb{E}_x\left[z_T(t\wedge {\tau}_T)\right].
\end{align}
Also, using submartingale property of $({z}_T(t))$, $t\geq 0$, and Doob optional stopping theorem, for any $t,h\geq 0$ and $x\in E$, we get
\begin{equation}\label{eq:th:w^t:mart2}
z_T(t\wedge {\tau}_T)\leq \mathbb{E}_x\left[z_T((t+h)\wedge {\tau}_T)|\mathcal{F}_t\right].
\end{equation}
Thus, we get $\mathbb{E}_x\left[z_T(t\wedge {\tau}_T)\right]\leq \mathbb{E}_x\left[z_T((t+h)\wedge {\tau}_T)\right]$, which combined with~\eqref{eq:th:w^t:mart1} shows $\mathbb{E}_x\left[z_T(t\wedge {\tau}_T)\right]= \mathbb{E}_x\left[z_T((t+h)\wedge {\tau}_T)\right]$ for any $t,h\geq 0$. Thus, we have equality in~\eqref{eq:th:w^t:mart2}, which concludes the proof.
\end{proof}

Now we are ready to show the proof of Proposition~\ref{pr:finite_hor}.

\begin{proof}[Proof of Proposition~\ref{pr:finite_hor}]

We present the proof only for $\underline{w}_T(x)$; the argument for $\overline{w}_T(x)$ is similar and is omitted for brevity. For transparency, we split the argument into three steps: (1) proof of monotonicity and continuity of $T\mapsto \underline{w}_T(x)$ for fixed $x\in E$; (2) proof of continuity of $x\mapsto \underline{w}_T(x)$ for fixed $T\geq 0$; (3) proof of joint continuity of $(T,x)\mapsto \underline{w}_T(x)$, optimality of $\underline{\tau}_T$ and martingale characterisation.

\medskip
\noindent {\it Step 1.} Monotonicity and continuity of $T\mapsto \underline{w}_T(x)$ for fixed $x\in E$. First, we prove monotonicity property of $T\mapsto \underline w_{T}(x)$.  Let $T,u\geq 0$ and let $\tau_{\vep}\leq T$ be an $\vep$-optimal stopping time for $e^{\underline{w}_{T}(x)}$. Then, using the fact that $g,G\geq 0$ we get
\begin{align}\label{eq:mono_underline_w}
e^{\underline{w}_{T-u}(x)} & \leq \E_x\left[e^{\int_0^{\tau_{\vep}\wedge (T-u)} g(X_s)ds + 1_{\{\tau_{\vep} <T-u\}}G(X_{\tau_{\vep}})}\right]\nonumber\\
&  \leq \E_x\left[e^{\int_0^{\tau_{\vep}} g(X_s)ds + 1_{\{\tau_{\vep} <T\}}G(X_{\tau_{\vep}})}\right] \leq e^{\underline{w}_{T}(x)} +\epsilon.
\end{align}
Letting $\vep\to 0$, we conclude that $T\mapsto \underline w_{T}(x)$ is non-decreasing. 

Second, we show continuity of $T\mapsto \underline{w}_T(x)$. Recalling that by Lemma~\ref{lm:opt_stop_finhor_bound} and Lemma~\ref{lm:fin1}, for any $x\in E$, the function $T\mapsto \underline{w}_T(x)$ is an increasing limit of continuous functions $T\mapsto \underline{v}_T^n(x)$, we get that $T\mapsto \underline{w}_T(x)$ is lower semicontinuous. This, together with the fact that $T\mapsto \underline{w}_T(x)$ is non-decreasing, shows left continuity of $T\mapsto \underline{w}_T(x)$. For the right continuity, let $\tau_{\epsilon}\leq T$ be an $\vep$-optimal stopping time for $e^{\underline{w}_T(x)}$.  Using monotonicity of $\underline w_{T}$, boundedness of $g$ and~\eqref{IF}, we get
\begin{align}
e^{\underline{w}_T(x)} \leq \lim_{u\downarrow 0}e^{\underline{w}_{T+u}(x)} & \leq \lim_{u\downarrow 0} \E_x\left[e^{\int_0^{\tau_\vep+u} g(X_s)ds + 1_{\{\tau_\vep+u<T+u\}}G(X_{\tau_\vep+u})}\right] \nonumber \nonumber\\
& = \E_x\left[e^{\int_0^{\tau_\vep} g(X_s)ds + 1_{\{\tau_\vep<T\}}G(X_{\tau_\vep})}\right]\leq e^{\underline{w}_T(x)}+\vep;\label{eq:w_down_left_final}
\end{align}
note in the second line we used bounded convergence theorem and the fact that $(X_t)$ is right continuous. Letting $\vep\to 0$ we get right continuity of $T \to \underline{w}_T(x)$, for any $x\in E$.

\medskip
\noindent {\it Step 2.} Continuity of $x\mapsto \underline{w}_T(x)$ for fixed $T\geq 0$. As in the first step, recalling Lemma~\ref{lm:opt_stop_finhor_bound} and Lemma~\ref{lm:fin1}, we get that, for any $T\geq 0$, the function $x\mapsto \underline{w}_T(x)$ is lower semicontinuous. To show upper semicontinuity we use dyadic approximation of $\underline{w}_T$. For any $m\in\bN$ and $T\geq 0$, we set
\begin{equation}\label{eq:1step.1}
\underline{w}_T^m(x):=\inf_{ \tau \in \mathcal{T}^m_T} \ln\E_x\left[ e^{\int_{0}^{\tau} g(X_s)ds+1_{\{\tau<T\}}G(X_\tau)}\right],\quad x\in E,
\end{equation}
where $\mathcal{T}^m_T$ is the family of stopping times taking values in $\left[0,{\tfrac{T}{2^m}},\tfrac{2T}{2^m}, \ldots,T\right]$. We show that, for any $T\geq 0$ and $m\in \mathbb{N}$, the map $x\mapsto \underline{w}_T^m(x)$ is continuous. Let us fix $T\geq 0$, $m\in\bN$, and define recursively the sequence of functions
\begin{align*}
\widetilde{w}_T^0(x) & :=0,\\
e^{\widetilde{w}_T^{j}(x)} & := \mathbb{E}_x\left[  e^{\int_0^{\frac{T}{2^m}} g(X_s)ds+\widetilde{w}_T^{j-1}(X_{\frac{T}{2^m}})}\right]\wedge e^{G(x)}, \quad j=1, \ldots, 2^m.
\end{align*}
By~\eqref{Feller}, the function $x\mapsto\widetilde{w}_T^{j}(x)$ is continuous, for $j=1,\ldots,2^m$. Also, using standard iteration arguments (see e.g. Section 2.2 in~\cite{Shi1978}) one can show that  $\underline{w}_T^m=\widetilde{w}_T^{2^m}$, which implies continuity of $x\mapsto \underline{w}_T^m(x)$.

We now show that $\lim_{m\to\infty} \underline{w}_T^m(x) = \underline{w}_T(x)$ for any $x\in E$ and $T\geq 0$. This together with continuity of $x\mapsto \underline{w}_T^m(x)$ and the fact that $(\underline{w}_T^m(x))_{m\in \mathbb{N}}$ is monotonically decreasing, shows upper semicontinuity of  $x\mapsto \underline{w}_T(x)$. Let $\vep>0$ and $\tau_\vep\leq T$ be an $\vep$-optimal stopping time for $e^{\underline{w}_T(x)}$. For any $m\in\bN$, we set
\[
\tau_{\vep}^m:=\inf\{\tau\in \cT_T^m \colon \tau\geq \tau_{\vep}\}= \textstyle \sum_{j=1}^{2^m}1_{\left\{\frac{T}{2^m}(j-1)< \tau_{\vep}\leq \frac{T}{2^m}j\right\}}\frac{T}{2^m}j.
\]
Noting that $\tau_{\vep}^m\leq T$, for any $x\in E$, we get
\begin{align}\label{cal1}
0&\leq  e^{\underline{w}_T^m(x)}-e^{\underline{w}_T(x)} \nonumber\\ 
&\leq  \E_x\left[ e^{\int_{0}^{\tau_{\vep}^m} g(X_s)ds+1_{\{\tau_{\vep}^m<T\}}G(X_{\tau_{\vep}^m})}\right] -\E_x\left[ e^{\int_{0}^{\tau_{\vep}} g(X_s)ds+1_{\{\tau_{\vep}<T\}}G(X_{\tau_{\vep}})}\right]+\vep \nonumber \\
&=  \E_x\left[ e^{\int_{0}^{\tau_{\vep}^m} g(X_s)ds}\left(e^{1_{\{\tau_{\vep}^m<T\}}G(X_{\tau_{\vep}^m})}-e^{1_{\{\tau_{\vep}<T\}}G(X_{\tau_{\vep}})}\right)\right] \nonumber \\
&\phantom{=} + \E_x\left[ e^{\int_{0}^{\tau_{\vep}} g(X_s)ds}\left(e^{\int_{\tau_{\vep}}^{\tau_{\vep}^m} g(X_s)ds}-1\right)e^{1_{\{\tau_{\vep}<T\}}G(X_{\tau_{\vep}})}  \right] +\vep  \nonumber \\
&\leq  \E_x\left[e^{\int_{0}^{\tau_{\vep}^m} g(X_s)ds}\left(e^{1_{\{\tau_{\vep}^m<T\}}G(X_{\tau_{\vep}^m})}-e^{1_{\{\tau_{\vep}<T\}}G(X_{\tau_{\vep}})}\right)\right]\nonumber\\
&\phantom{=} + \left(e^{\underline{w}_T(x)}+\varepsilon\right)\left(e^{\frac{T}{2^m}\Vert g\Vert}-1\right)+\vep.
\end{align}
For any $T\geq 0$ and $x\in E$, we get $\left(e^{\underline{w}_T(x)}+\varepsilon\right)\left(e^{\frac{T}{2^m}\Vert g\Vert}-1\right)\to 0$ as $m\to \infty$. Also, noting that $\tau_{\vep}^m \downarrow \tau_{\vep}$ and using~\eqref{IF}, we get 
\begin{align}\label{cal2}
& \E_x\left[e^{\int_{0}^{\tau_{\vep}^m} g(X_s)ds}\left(e^{1_{\{\tau_{\vep}^m<T\}}G(X_{\tau_{\vep}^m})}-e^{1_{\{\tau_{\vep}<T\}}G(X_{\tau_{\vep}})}\right)\right] \nonumber\\
& \phantom{=} \leq \E_x\left[e^{\int_{0}^{\tau_{\vep}^m} g(X_s)ds}\left(e^{1_{\{\tau_{\vep}<T\}}G(X_{\tau_{\vep}^m})}-e^{1_{\{\tau_{\vep}<T\}}G(X_{\tau_{\vep}})}\right)\right] \nonumber\\
& \phantom{=} \leq e^{T\Vert g\Vert}\E_x\left|e^{G(X_{\tau_{\vep}^m})}-e^{G(X_{\tau_{\vep}})}\right|\to 0, \quad m\to\infty.
\end{align}
Consequently, letting $\epsilon\to 0$ in \eqref{cal1}, we conclude the proof of this step.

\medskip
\noindent {\it Step 3.} Continuity of $(T,x)\mapsto \underline{w}_T(x)$, optimality of~\eqref{optstop}, and martingale characterisation. Let the sequence $(T_n)\subset \mathbb{R}_+$ be monotone and such that $T_n\to T$, and $(x_n)\subset E$ be such that $x_n\to x\in E$. Using continuity of $x\mapsto\underline{w}_T(x)$ and monotonicity of $T\mapsto\underline{w}_T(x)$, from Dini's theorem we get that the convergence of $\underline{w}_{T_n}(x)$ to $\underline{w}_T(x)$ is uniform in $x$ from compact sets; see Theorem 7.13 in~\cite{Rud1976} for details. Thus, we get
\begin{equation}\label{eq:th:w^t:conv_t_x}
|\underline{w}_{T_n}(x_n)-\underline{w}_T(x)|\to 0, \quad n\to\infty,
\end{equation}
which shows continuity of the map $(T,x)\mapsto \underline{w}_T(x)$. Thus, using Lemma~\ref{lm:opt_fin} we get that, for any $T\geq 0$ and $x\in E$, the stopping time $\underline{\tau}_T$ is optimal for $\underline{w}_T(x)$, the process $\underline{z}_T(t)$ is a $\mathbb{P}_x$-submartingale and $\underline{z}_T(t\wedge \underline{\tau}_T )$ is a $\mathbb{P}_x$-martingale, which concludes the  proof.
\end{proof}

\section*{Acknowledgements}
Damian Jelito and {\L}ukasz Stettner acknowledge research support by NCN grant no. 2020/37/B/ST1/00463.


\end{document}